\documentclass{amsart}

\usepackage{lineno}
\usepackage[backref=page]{hyperref}
\usepackage{amsthm, amsmath, amssymb}
\usepackage{mathrsfs}
\usepackage[msc-links, backrefs, nobysame, initials]{amsrefs}
\usepackage[all]{xy}
\usepackage{calrsfs}
\usepackage{mathpazo}
\usepackage{xcolor}

\usepackage{graphicx}
\usepackage{caption}
\usepackage{subcaption}

\newcommand{\R}{\mathbb R}

\newcommand{\Z}{\mathbb Z}

\newcommand{\M}{\mathcal M}

\newcommand{\X}{\mathfrak X}
\newcommand{\Y}{\mathfrak Y}

\DeclareMathOperator{\diff}{Diff} 
\DeclareMathOperator{\pdiff}{PDiff}
\DeclareMathOperator{\mcg}{Mod}
\DeclareMathOperator{\pmcg}{PMod}
\DeclareMathOperator{\Smcg}{SMod}

\DeclareMathOperator{\Lmcg}{LMod}

\DeclareMathOperator{\Ima}{Im}

\DeclareMathOperator{\vcd}{vcd}

\DeclareMathOperator{\phome}{PHomeo}
\DeclareMathOperator{\Aut}{Aut}

\DeclareMathOperator{\sig}{sig}
\DeclareMathOperator{\SK}{SK}
\DeclareMathOperator{\Area}{Area}
\DeclareMathOperator{\Iso}{Isom}
\DeclareMathOperator{\cov}{Cov}
\DeclareMathOperator{\Smodn}{S\mathcal{N}_g^k}

\newcommand{\SKp}{\SK_p(h,t)}         
\newcommand{\SKtopk}{\mathcal{SK}_p^k(h,t)} 
\newcommand{\Ttuples}{\mathscr{B}(t)}       
\newcommand{\Ttuplesk}{\mathscr{B}^k(t)}    
\newcommand{\modn}{\mathcal{N}_g^k}     
\newcommand{\modnw}{\mathcal{N}_{g}}    
\newcommand{\modno}{\mathcal{N}_g^1}    
\newcommand{\modnpk}{\mathcal{N}_{p}^{k}}
\newcommand{\modnk}{\mcg(N_g;k)}    
\newcommand{\modsg}{\mcg(S_g)}      
\newcommand{\modng}{\mcg(N_g)}      
\newcommand{\hiper}{\mathbb{H}^2}   
\newcommand{\Isohiper}{\Iso(\hiper)}   
\newcommand{\GC}{\mathrsfs{C}}          
\newcommand{\CClk}{\GC_p^k(h,t)}    
\newcommand{\Zp}{\Z/p}                  
\newcommand{\Farr}[2]{\widehat{H}^{#1}(#2;\Z)}  
\newcommand{\Farrp}[2]{\widehat{H}^{#1}(#2;\Z)_{(p)}}  

\newcommand{\w}[1]{\widetilde{#1}}  
\newcommand{\wk}[1]{\widehat{#1}}  

\newcommand{\diffnk}{\diff(N_g;k)}

\newcommand{\pdiffnk}{\pdiff(N_g;k)}

\newcommand{\pmodsk}{\pmcg(S_g;k)}

\newcommand{\Smodnk}{\Smcg(N_g;k)}

\newcommand{\Lmodnht}{\Lmcg(N_h;t)}

\newcommand{\n}{\noindent}

\newtheorem{thm}{Theorem}[section]
\newtheorem{lem}[thm]{Lemma}
\newtheorem{prop}[thm]{Proposition}
\newtheorem{cor}[thm]{Corollary}

\newtheorem{mainthm}{Theorem}

\theoremstyle{definition}
\newtheorem{defi}[thm]{Definition}
\newtheorem{cons}[thm]{Construction}
\newtheorem{notation}[thm]{Notation}

\theoremstyle{remark}
\newtheorem{rmk}[thm]{Remark}
\newtheorem{example}[thm]{Example}
\newtheorem*{claim}{Claim}

\numberwithin{equation}{section}

\newcommand{\cuadro}[8]{\xymatrix @C=1.3cm @M=2mm{
	{#1} \ar[r]^-{#2} \ar[d]_-{#4} & {#3} \ar[d]^-{#5}  \\
	{#6} \ar[r]^-{#7} & {#8}   }}

\begin{document}

\title[Farrell cohomology of the pure MCG of non-orientable surfaces]{Farrell cohomology of the pure mapping class group of non-orientable surfaces}

\author{Nestor Colin}
\address{Instituto de Matem\'aticas, Universidad Nacional Aut\'onoma de M\'exico
Oaxaca de Ju\'arez, Oaxaca 68000, M\'exico}
\email{ncolin@im.umam.mx}






\begin{abstract}
For an odd prime $p$, we determine the $p$-primary component of the Farrell cohomology of the pure mapping class groups of a non orientable surface of genus $p$ with $k\geqslant 1$ marked points. To do this, we classify conjugacy classes of subgroups of order $p$ of the pure mapping class group of a non orientable surface of any genus with marked points. This is obtained by extending the notion of topological equivalence for surface kernel epimorphisms of non Euclidean crystallographic groups, adapting it to the setting of surfaces with marked points. 
\end{abstract}

\maketitle



\section{Introduction}

Let $N_g$ be a closed non-orientable surface of genus $g$ and $\{z_1, \dots, z_k\}$ a set of $k\geqslant 0$ distinct points in $N_g$, we call them {\it marked points}. Let $\diff(N_g ; k)$ be the group of diffeomorphisms of $N_g$ that preserve the set of marked points and let $\pdiff(N_g ; k)$ be the subgroup of diffeomorphisms that fix the marked points pointwise. 
The {\it pure mapping class group} of $N_g$ with $k$ marked points is the group of isotopy classes of $\pdiff(N_g; k)$.
We use the notation $\mathcal{N}_g^k:=\pmcg(N_g;k)$. If the set of marked points is empty, we omit $k$ from the notation.

Studying the cohomology of mapping class groups has been a central topic in mathematics due to its rich connections with other areas, such as low-dimensional topology and moduli spaces. One method for studying this cohomology in higher degrees is through Farrell cohomology (see \cites{Xia92, Xia92Farr2,GMX92, Mis94, Xia90The,  GloQJ07, LuFarrell}). This cohomology is defined for groups $\Gamma$ with finite virtual cohomological dimension ($\vcd$), and in degrees above the $\vcd$, it agrees with the ordinary cohomology of the group. An important aspect that facilitates the computation of Farrell cohomology is its periodicity. Recall that a group $\Gamma$ with finite $\vcd$ is called $p$-periodic if there exist a non-zero integer $d$ such that the $p$-primary component of the Farrell cohomology groups $\Farr{i}{\Gamma}$ and $\Farr{i+d}{\Gamma}$ are isomorphic for all $i$. In this situation the $p$-primary component it becomes well-understood through Brown's formula \cite{Brown82Coh}*{Corollary X.7.4}:
\begin{equation}\label{Eqn:Browns:Formula}
    \Farrp{*}{\Gamma}\cong \prod_{\Zp\in S} \Farrp{*}{N_\Gamma(\Zp)}
\end{equation}
where $S$ is a set of representatives of conjugacy classes of subgroups of order $p$ and $N_\Gamma(\Zp)$ is the normalizer of $\Zp$ in $\Gamma$. 


The main objective of this work is to develop tools to provide explicit calculations of the $p$-primary component of Farrell cohomology of the pure mapping class group of a non-orientable surface $\modn$.

It is well known that $\modn$ has finite virtual cohomological dimension and has been calculated explicitly in \cite{Iv87}*{Theorem 6.9}, and recently, in \cite{CJX24Per}, N. Colin, R. Jiménez Rolland, and M. A. Xicoténcatl proved that if $g\geqslant 3$ and $k\geqslant 1$, the group $\modn$ has $p$-periodic cohomology whenever the group $\modn$ has $p$-torsion. Furthermore, they proved that its $p$-period is equal to $4$. This result combined with Equation \ref{Eqn:Browns:Formula} provides a clear method for determine the $p$-primary component of the Farrell cohomology of $\modn$ because reduce the problem into three fundamental steps:

\begin{enumerate}
    \item Finding subgroups of order $p$ in $\modn$.
    \item Classifying conjugacy classes of subgroups of order $p$ in $\modn$.
    \item Studying the cohomology of normalizers of subgroups of order $p$ in $\modn$.
\end{enumerate}

In this work, we explore each of these steps for the pure mapping class group of a non-orientable surface $\modn$ with $k \geqslant 1$ marked points, with the objective of providing explicit computations of the Farrell cohomology in some specific cases. We begin with the existence of subgroups of order $p$ in $\modn$. Using the Nielsen realization theorem for non-orientable surfaces \cite{CX24Nil}, we construct a branched cover, which is associated with a Riemann-Hurwitz equation. We show that the existence of $p$-torsion in $\modnw$ is analogous to finding solutions to the Riemann-Hurwitz equation. A Birman exact sequence argument provides a characterization of the existence of $p$-torsion in $\modn$ for $k \geqslant 1$, leading to the following result.

\begin{mainthm}\label{Thm:Main:Torsion}
    Let $g\geqslant 3$ and $k\geqslant 1$. Then $\modn$ contains a subgroup of order $p$ if and only if the Riemann-Hurwitz equation $g-2=p(h-2)+t(p-1)$ has an integer solution for $h$ and $t$ with $t\geq k, h\geq 1$.
\end{mainthm}

This result, combined with \cite{CJX24Per}*{Theorem 1}, provides conditions under which $\modn$ has $p$-periodic cohomology. Next, in Section \ref{Sec:Conjugacy} we study surface kernels  $\theta:\Gamma\to \Zp$ of a non-Euclidean crystallographic group $\Gamma$ (NEC group) by introducing the notion of topological equivalence of surface kernels relative to $k$ marked elliptic generators. This notion is crucial because it provides the correct outline to determine the conjugacy classes of $\modn$ with $k\geqslant 1$. Essentially, marking elliptic generators in $\Gamma$ is analogous to marking points on the surface. In this context, we prove the following result.

\begin{mainthm}\label{Thm:Main:CClass:SK}
    Let $g\geqslant 3$, $k\geqslant 1$, $h\geqslant k$ and $t\geqslant 1$ be a solution to the Riemann-Hurwitz equation. Then the conjugacy classes of subgroups of order $p$ in $\modn$ that act on $N_g$ with $t$ fixed points are in one-to-one correspondence with the set of surface kernels $\theta: \Gamma \to \Zp$ up to topological equivalence relative to the $k$ marked elliptic generators, where $\Gamma$ is a NEC group such that $\sig(\Gamma)=(h;-;[(p)^t])$.
\end{mainthm}

This result reduces the problem of finding conjugacy classes of subgroups of order $p$ in $\modn$ into a combinatorial one. Specifically, each surface kernel $\theta: \Gamma \to \Zp$ corresponds to a $t$-tuple $(\theta(x_1), \ldots, \theta(x_k) \mid \theta(x_{k+1}), \ldots, \theta(x_t)) \in (\Zp)^t$, where $x_1, \ldots, x_k$ are the marked elliptic generators and $x_{k+1}, \ldots, x_t$ are the remaining elliptic generators of $\Gamma$. By introducing the notion of congruence in the set of $t$-tuples, we establish the following result based on the previous theorem.

\begin{mainthm}\label{Thm:Main:CClass:tTupl}
    Let $g\geqslant 3$, $k\geqslant 1$, $h\geqslant k$ and $t\geqslant 1$ be a solution to the Riemann-Hurwitz equation. Then the conjugacy classes of subgroups of order $p$ in $\modn$ that act on $N_g$ with $t$ fixed points are in one-to-one correspondence with the set of $t$-tuples $(1,\beta_2, \ldots, \beta_k \mid \beta_{k+1}, \ldots, \beta_t)$ up to congruence equivalence.
\end{mainthm}

The final step, determining the cohomology of normalizers, is not yet complete for all cases and genera $g$. This requires a more specific treatment according to the particular characteristics of $\modn$ and its conjugacy classes of subgroups of order $p$. In general, using Birman-Hilden theory, for any given subgroup $\Zp \leqslant \modn$ there exists an extension of groups:
\[
    1\to \Zp \to N(\Zp) \to H \to 1
\]
where $h \geqslant 1$ and $t \geqslant k$ are solutions to the Riemann-Hurwitz equation associated with the group $\Zp$, and $H$ is a subgroup of $\mcg(N_h;t)$. Particularly, in the case where $g = p$, this extension, combined with the fact that the $p$-period of $\modnpk$ is equal to $4$ (see \cite{CJX24Per}*{Theorem 2}), leads to the following result.

\begin{mainthm}\label{Thm:Main:Nor:Npk}
    Let $p$ be an odd prime and $\Zp \leqslant\mathcal{N}_p^k$, with $k=1,2$. Then $N(\Zp)\cong D_{2p}$ where $D_{2p}$ is the dihedral group of order $2p$. Furthermore, we have that
    $$
        \Farrp{i}{N(\Zp)}= \begin{cases}
            \Zp 	& 	i\equiv 0    \mod (4) \\
            0 		&	i\equiv 1,2,3   \mod (4).
        \end{cases}
    $$
\end{mainthm}

Finally, at the end of this paper, we provide an explicit computation of the Farrell cohomology using the techniques and results developed throughout our work:

\begin{mainthm}\label{Thm:Main:FarrellNpk}
    Let $p$ be an odd prime. Then:
    \begin{align*}
        \Farrp{i}{\modnpk}=& \begin{cases}
        \left( \Zp \right)^{\tfrac{p-1}{2}} & i\equiv 0 \mod (4) \\
         0                 & \text{other cases}
    \end{cases}    
    & \text{for } & k=1,2; \\
    \\
        \Farrp{i}{\modnpk} = & \ \ 0    & \text{for } & k\geqslant 3.
    \end{align*}
\end{mainthm}

Results of the $p$-primary component of the Farrell cohomology 
in the orientable case $\Gamma_g^k=\pmodsk$ were obtained in \cites{GloQJ07, Xia92Farr2, Xia95, lu01per} where the authors determined 
$\widehat{H}^*(\Gamma_{p-1}^k; \Z)_{(p)}$, $\widehat{H}^*(\Gamma_{(p-1)/2}^k; \Z)_{(p)}$, $\widehat{H}^*(\Gamma_{p}; \Z)_{(p)}$, $\widehat{H}^*(\Gamma_{k(p-1)/2}^{k+1}; \Z)_{(p)}$ and $\widehat{H}^*(\Gamma_{k(p-1)/2}^{k+2}; \Z)_{(p)}$ with $k\geqslant 0$ and $p$ is an odd prime. Furthermore, Q. Lu \cite{LuFarrell} obtained all the $p$-primary components of the Farrell cohomology of the pure mapping class group $\Gamma_g^k$ of a surface of low genus $g=1,2,3,$ when $\Gamma_g^k$ has $p$-torsion, $p$ is an odd prime, and $k\geqslant 1$.  In the closed case of a non-orientable surface, G. Hope and U. Tillman, in \cite{HT09}, investigated the $p$-periodicity of the Farrell cohomology of $\mathcal{N}_g$  where $g \geqslant 3$ establishing some conditions for which $\mathcal{N}_g$ has $p$-periodic cohomology. However, their methods were not sufficient for computing some of the $p$-primary components of these groups even to compute the $p$-period.

\subsection*{Outline.} In Section \ref{Sec:Subgroups} we study subgroups of order $p$ in $\modn$ proving Theorem \ref{Thm:Main:Torsion}. In Section \ref{Sec:Conjugacy} we recall some facts about non-euclidean crystallographic subgroups and surface kernels epimorphism for this groups. After recalling these objects, we introduce the notion of topological equivalence relative to the $k$ marked elliptic generators, which allows us to connect with the conjugacy classes of subgroups of order $p$ in $\modn$, proving Theorem \ref{Thm:Main:CClass:SK} and Theorem \ref{Thm:Main:CClass:tTupl}. In Section \ref{Sec:Normalizers} we review some notation about Birman-Hilden theory in the context of non-orientable surfaces, focusing on normalizers of subgroups of order $p$ in $\modn$. We prove that the normalizer $N(\Zp)$ of a subgroup $\Zp \leqslant \modnpk$ is equal to the dihedral group $D_{2p}$ of order $2p$ (Theorem \ref{Thm:Main:Nor:Npk}). Finally, in Section \ref{Sec:FarrellNpk}, we obtain the $p$-primary component of the Farrell cohomology of the pure mapping class group $\modnpk$, proving Theorem \ref{Thm:Main:FarrellNpk}.  




\section{Subgroups of order \texorpdfstring{$p$}{p} in \texorpdfstring{$\modn$}{N g k}}
\label{Sec:Subgroups}

In \cite{CJX24Per}*{Theorem 1}, it was shown that the pure mapping class group $\modn$ of a non-orientable surface has $p$-periodic cohomology for $g \geqslant 3$ and $k \geqslant 1$ whenever $\modn$ contains $p$-torsion.  
In this section, we provide conditions under which subgroups of order $p$ exist in the pure mapping class group $\modn$, leading to cases where its cohomology is $p$-periodic.
These conditions are derived by the classical Riemann-Hurwitz equation applied to a branched covering space, which arises from the action of $\Zp$ on the surface $N_g$.

For the rest of the paper, we assume that $g \geqslant 3$, unless otherwise stated. We begin with a similar result obtained by Q. Lu in \cite{lu01per}*{Theorem 2.7} for the orientable case.

\begin{prop}
    If $\modnw$ contains a subgroup of order $p$, then the equation 
    \[
        g-2=p(h-2)+t(p-1) 
    \]
    has a non-negative solution $h\geq 1 $.
\end{prop}
\begin{proof}
    By the Nielsen realization theorem \cite{CX24Nil}*{Theorem} there exist a group $G\leqslant \diff(N_g)$ such that $\pi(G)=\Zp$. Thus, $G$ acts on the surface $N_g$ and we obtain a branched cover $q:N_g \to N_g/G$. Let $t$ be the branched points of the covering $q$ and $h$ be the genus of the surface $N_g/G$, then, by Riemann-Hurwitz equation the result follows. 
\end{proof}

The converse of the previous statement holds if $t$ is greater than zero and and its proof relies on surface kernels of non-Euclidean crystallographic groups (NEC groups), which we will review in the following section.

\begin{prop}\label{Prop:RH and ptorsion}
    If $g-2=p(h-2)+t(p-1)$ has an integer solution $(h,t)$ with $h,t\geqslant 1$, then $\modnw$ contains a subgroup of order $p$.
\end{prop}
\begin{proof}
   Let $\Gamma$ be the NEC group whose algebraic presentation is:
   $$\Gamma=\langle d_1, \ldots , d_h , x_1, \ldots , x_t \mid x_1^p=x_2^p=\ldots = x_t^p = x_1 \cdot \ldots \cdot x_t \cdot d_1^2 \cdot \ldots \cdot d_h^2=1 \rangle. $$
   Consider the presentation of the group $\Zp=\langle y \mid y^p =1 \rangle$ and define the function $\theta : \Gamma \to \Zp$ given at level of the generators as follows:
   \begin{align*}
       \theta(x_i)& =\begin{cases}
           y^{1+\epsilon} & \text{si } i=1 \\
           y  & \text{si } i=2, \ldots ,t.
       \end{cases} 
       & \theta(d_j)& =\begin{cases}
           y^{\tfrac{-t-\epsilon}{2}} & \text{si } j=1 \\
           1  & \text{si } j=2, \ldots ,h.
        \end{cases}
   \end{align*}
   where $\epsilon=1$ if $t$ is odd or $\epsilon=0$ in other case. Since $\theta:\Gamma\to \Zp$ preserves all the relations of the group $\Gamma$, $\theta$ is a homomorphism that is also surjective. Furthermore, since $\ker (\theta)$ is a normal subgroup of $\Gamma$, it is also a NEC group. 
   By \cite{Bujalance10}*{Theorem 1.2.2}, since the order of the image of the elliptic generators $x_i$ is $p$, the proper periods of $\ker(\theta)$ are equal to one, as $p$ is prime. Thus, there are no elliptic generators in $\ker(\theta)$, and it follows that $\ker(\theta)$ is a surface NEC group. Using Riemann-Hurwitz equation for NEC groups (see \cite{Bujalance10}*{Theorem 1.1.8} and Equation \ref{Eqn:RH:NEC:Groups}) and since $g - 2 = p(h - 2) + t(p - 1)$, we conclude that $\ker(\theta)$ is isomorphic to the fundamental group $\pi_1(N_g)$. Consequently, $\theta: \Gamma \to \Zp$ defines a surface kernel of NEC groups, and by \cite{Singerman71}*{Theorem 1}, $\Zp$ is a subgroup of $\Aut(N_g; \X)$. Therefore, $\Zp$ is also a subgroup of $\modng$.
\end{proof}

Now, we prove Theorem \ref{Thm:Main:Torsion}.

\begin{proof}[Proof of Theorem \ref{Thm:Main:Torsion}]
    We proceed by induction on the number of marked points. Suppose that $\modno$ contains a subgroup of order $p$. The Birman exact sequence for non-orientable surfaces (see \cite{Gram73}*{Proposition 1 and Lemma 1} and \cite{Kork02}*{Theorem 2.1}) provides the following
    $$ 1 \to \pi_1 (N_g) \rightarrow \modno \rightarrow \modnw \to 1.$$
    Since $\pi_1 (N_g)$ is a torsion free group, it follows that $\modnw$ must contain $p$-torsion. Thus, by Proposition \ref{Prop:RH and ptorsion} there exist $h,t\geqslant 1$ such that the Riemann-Hurwitz equation holds,
    which proves the result for $k=1$. Now, suppose the result holds for $\modn$. Using the Birman exact sequence
        $$ 
        1 \to \pi_1 (N_g^{k}) \rightarrow \mathcal{N}_g^{k+1} \rightarrow \mathcal{N}_g^{k} \to 1,
        $$    
    and applying the induction hypothesis, we prove the result for $\mathcal{N}_g^{k+1}$. 
    The converse can be proven similarly to the previous proposition and we omit it. 
\end{proof}


\begin{example}\label{Ex:genus_p} 
We describe some surfaces which has subgroups of order $p$ according to the previous results.

$\vartriangleright$ \textit{Genus ${g=p}$}. 
The equation $g-2 = p(h-2) + t(p-1)$ has the unique solution $(h, t) = (1, 2)$. Therefore, the pure mapping class group $\modnpk$ contains a subgroup of order $p$ for $k = 0, 1, 2$. Additionally, every diffeomorphism of order $p$ on $N_p$ has exactly $2$ fixed points, and the quotient of $N_p$ under the action of this diffeomorphism is $\mathbb{R} P^2$. For $k \geqslant 3$, the pure mapping class group $\modnpk$ does not admit a subgroup of order $p$.

$\vartriangleright$ \textit{Genus ${g=p+1}$}. 
In this case, the equation $g-2 = p(h-2) + t(p-1)$ has the unique solution $(h, t) = (2, 1)$. Thus, the pure mapping class group $\mathcal{N}_{p+1}^k$ has a subgroup of order $p$ for $k = 0, 1$. Consequently, every diffeomorphism of order $p$ on $N{p+1}$ has $1$ fixed point, and the quotient space under the diffeomorphism is a Klein bottle. For $k \geqslant 2$, $\mathcal{N}_{p+1}^k$ does not have a subgroup of order $p$.

$\vartriangleright$ \textit{Genus ${g=2p-1}$}. 
The equation $g-2 = p(h-2) + t(p-1)$ has a unique solution $(h, t) = (1, 3)$. In this case, the pure mapping class group $\mathcal{N}_{2p-1}^k$ admits a subgroup of order $p$ for $k = 0, 1, 2, 3$. However, for $k \geqslant 4$, no such subgroup exists in $\mathcal{N}_{2p-1}^k$.
\end{example}

\section{Conjugacy classes of subgroups of order \texorpdfstring{$p$}{p} in \texorpdfstring{$\modn$}{Ngk}}
\label{Sec:Conjugacy}

In this section, we recall some theory about non-Euclidean crystallographic groups (NEC groups) and their surface kernels epimorphism. We also review the notion of topological equivalence of surface kernels and extend this notion to include surfaces with $k$ marked points $N_g$. Using Nielsen realization theorem a surface kernel is defined for a subgroup of order $p$ of $\modn$, and the notion introduced, combined with the geometric realization of NEC group isomorphisms, leads to a one-to-one correspondence between surface kernels (up to equivalence) and the conjugacy classes of subgroups of order $p$ in $\modn$.

\subsection{Non-euclidean crystallographic groups.} By a non-euclidean crystallographic group (NEC-group for short) we mean a discrete and co-compact subgroup of the group of all isometries of the hyperbolic plane $\hiper$. It turns out that the algebraic structure of such a group $\Gamma$ are encoded in its signature:
\begin{equation}\label{Eqn:Signature}
    \sig(\Gamma):=(g;\pm ;[m_1,\ldots, m_t])
\end{equation}
where $m_i$ are called the proper periods and $g$ is the orbit genus of $\Gamma$. If $m_i=m$ for all $i$, we write $\sig(\Gamma)=(g;\pm;[(m)^t])$. The signature determines a presentation of the group $\Gamma$ by means of generators $x_i$ for $i=1,\ldots,t$; $a_j, b_j$ if $\sig(\Gamma)$ has sign $+$ or $d_j$ if has sign $-$ for $j=1\ldots, g$. These generators satisfy the following relations: 
\begin{align*}
     & x_i^{m_i} =1 & & \text{for }  i=1,\ldots , t; \\
     & x_1\cdot \ldots \cdot x_t\cdot [ a_1, b_1] \cdot  \ldots \cdot [a_g,b_g]  =1 & & \text{if } \sig(\Gamma) \text{ has sign } + \\
     & x_1\cdot \ldots \cdot x_t\cdot d_1^2\cdot \ldots \cdot d_g^2  =1 & & \text{if } \sig(\Gamma) \text{ has sign } - 
\end{align*}
    where $[a,b]=a \cdot b \cdot a^{-1}\cdot b^{-1} .$
    The isometries $x_i$ are elliptic transformations,  $a_j$, $b_j$ are hyperbolic translations and $d_j$ are glide reﬂections. The orbit space $\hiper/\Gamma$ is a hyperbolic orbifold of genus $g$ and it is orientable if the sign is $+$ and non-orientable otherwise. The image on $\hiper/\Gamma$ of the fixed point of the elliptic generator $x_i$ is called cone point of order $m_i$. Every NEC group $\Gamma$ with signature \ref{Eqn:Signature} has associated a fundamental region whose area $\Area(\Gamma)$, called the area of the group, is:
    \[
        \Area(\Gamma)= 2\pi \left( \eta g -2 + \sum_{i=1}^{t} \left( 1-\frac{1}{m_i} \right) \right),
    \]
    where $\eta=1$ if the sign is $-$ or $\eta=2$ otherwise. Notice that this definition can be used to an abstract group $\Gamma$ with the previous presentation, and such a group $\Gamma$ can be realized as a NEC group if and only if $\Area(\Gamma)>0$. 

    If $K$ is a subgroup of a NEC group $\Gamma$ of finite index, then also $K$ is a NEC group and the following Riemann-Hurwitz formula holds:
    \begin{equation}\label{Eqn:RH:NEC:Groups}
        \Area(K)=[\Gamma:K] \cdot \Area(\Gamma)
    \end{equation}
    
    Torsion free NEC group $\Gamma$ is called \textit{surface NEC group} and NEC groups without orientation reversing elements are \textit{Fuchsian groups}. 

    Recall that if a NEC group $\Gamma$ has signature equal to $(g;-;[m_1,\ldots, m_t])$, then the group $\Gamma$ has the following presentation 
    \begin{equation} \label{Eqn:Canonical:Presentation}
        \langle d_1, \ldots, d_h , x_1, \ldots, x_t \mid x_1^{m_1}= \ldots = x_t^{m_t} =  x_1 \cdot \ldots \cdot x_t \cdot d_1^2 \cdot \ldots d_h^2=1 \rangle.
    \end{equation}  
    We call the above presentation \textit{a canonical presentation} of $\Gamma$ and the elements $x_i$ and $d_j$ are called \textit{canonical generators.} 

\subsection{Surface kernel epimorphisms of NEC groups.} Let $\Gamma$ be a NEC group, $G$ be a finite group and $\theta: \Gamma \to G$ be an epimorphism. We say that $\theta$ is a \textit{surface kernel epimorphism}
if and only if $\ker(\theta)$ is a surface NEC group. 

Two surface kernel epimorphisms $\theta_i:\Gamma_i\to G_i,$ $i=1,2,$ are \textit{topologically equivalent} if and only if there exist isomorphism $\varphi: \Gamma_1\to \Gamma_2$ and $\psi:G_1 \to G_2$ such that the following diagram is commutative
\[  
\xymatrix{
\Gamma_1 \ar[r]^-{\theta_1} \ar[d]_-{\varphi}  &   G_1 \ar[d]^-{\psi}  \\
\Gamma_2 \ar[r]_-{\theta_2}     &   G_2
}
\]

Surface kernels of Fuchsian groups classify actions of finite groups on orientable surfaces $S_g$, see for instance \cite{Broug91}*{Proposition 2.2}, \cite{Broughton22Equiv} and \cite{Brouhton22Fut}. This classification, along with Nielsen's realization theorem, connects surface kernels to conjugacy classes of finite subgroups of $\modsg$. One of our objectives is to extend these results to the context of non-orientable surfaces $N_g$ with $k$ marked points and apply them to the case of subgroups of order $p$ of $\modn$. First, we extend the notion of surface kernels to include pure mapping class groups with marked points $\modn$. This notion will be useful in Theorem \ref{Thm:Conjugacy:Classes:Surface:Kernels} in order to prove that there is a one-to-one correspondence between surface kernels and conjugacy classes of subgroups of order $p$ in $\modn$. This correspondence can also be extended to finite subgroups of the pure mapping class group $\modn$ or the full mapping class group $\modnk$ where marked points can be permuted by the diffeomorphism. However, these extensions are not pursued here, as they are not required for the results and applications discussed. It is also important to note that the results of this section apply in the case where $k=0$, that is, to $\modnw$.

Note that by Theorem \ref{Thm:Main:Torsion}, the existence of $p$-torsion in $\modn$ is analogous to finding solutions to the Riemann-Hurwitz equation. Considering this, and given that our focus is on studying conjugacy classes of subgroups of order $p$ through surface kernels, we introduce the following notation.

\begin{notation}
    Let $p$ be an odd prime and $h,t\geqslant 1$ be a solution to Riemann-Hurwitz equation 
    $
        g-2=p(h-2)+t(p-1).
    $
    We denote by $\SKp$ the set formed by all the surface kernels $\theta: \Gamma \to G $ such that $G\cong \Zp$, $\ker(\theta)\cong \pi_1(N_g)$ and $\sig(\Gamma)=(h;-;[(p)^t]).$
\end{notation} 

Recall that if a NEC group $\Gamma$ has a signature $(h; -; [(p)^t])$, then the group $\Gamma$ has the following presentation:
\begin{equation}\label{Eqn:Canonical:Pres:Gammap}
    \langle d_1, \ldots, d_h , x_1, \ldots, x_t \mid x_1^p = \ldots = x_t^p =  x_1 \cdot \ldots \cdot x_t \cdot d_1^2 \cdot \ldots \cdot d_h^2 = 1 \rangle.
\end{equation}    
For the rest of this paper, we fix an abstract group $\Gamma$ with the above presentation, for two key reasons. First, the action of a subgroup of order $p$ in $\modn$ yields a specific Riemann-Hurwitz solution $(h, t)$. To understand the conjugacy classes of this subgroup, we study the set of surface kernels $\SKp$ corresponding to this solution, where the NEC groups of the surface kernels has the same signature as $\Gamma$. Second, specifying a set of marked points on the surface will be analogous to marking the first $k$ elliptic generators of the NEC group. Therefore, having a fixed presentation of the group $\Gamma$ with $k$ marked elliptic generators is useful for maintaining control of the elliptic generators over all NEC groups with the same signature.

\begin{rmk}
    If $\theta: \Gamma' \to G$ belongs to $\SKp$, then there is an isomorphism between $\Gamma$ and $\Gamma'$. By choosing an isomorphism $G \to \Zp$, $\theta$ is topologically equivalent to a surface kernel $\theta': \Gamma \to \Zp$. In this way, surface kernels in $\SKp$ can be regarded, up to topological equivalence, as having a common source $\Gamma$ and target $\Zp$.
\end{rmk}

\begin{defi} \label{Def:SurfaceKernel:Marked}
    Let $\theta_1: \Gamma \to \Zp $ and $\theta_2: \Gamma \to \Zp$ be two surface kernels and $h\geqslant 1, t\geqslant 1$ such that $(h,t)$ satisfies the Riemann-Hurwitz equation \eqref{Eqn:Riemann-Hurwitz} and  $\theta_1,\theta_2\in \SKp$. Let $k$ be a non-negative integer with $k\leqslant t$.   We say that $\theta_1$ and $\theta_2$ \textit{are topologically equivalent relative to the $k$ marked elliptic generators} if there exist two isomorphisms $ \varphi: \Gamma \to \Gamma $ and $ \psi : \Zp \to \Zp $ such that $\varphi (x_i) $ belongs to the conjugacy class of some non-trivial power of $ x_i$ for each $i=1,\ldots, k$ and the following diagram is commutative:
    \begin{equation}\label{Eqn:Diagram_equivalence_of_SK}
        \xymatrix{
        \Gamma \ar[r]^-{\theta_1} \ar[d]_-{\varphi} & \Zp \ar[d]^-{\psi} \\
        \Gamma \ar[r]^-{\theta_2} & \Zp 
        }
    \end{equation}
    We denote by $\SKtopk$ the set of surface kernels up to topologically equivalence relative to the $k$ marked elliptic generators.
\end{defi}

\begin{rmk}\label{Rmk:conjugated_to_x_or_x-1}
   {Given an isomorphism $\varphi:\Gamma \to \Gamma$, then $\varphi(x_i)$ is conjugated to $x_j$ or $x_j^{-1}$ for each $i=1,\ldots, t$, (see \cite{Bujalance15}*{Theorem 1}). Particularly, given two surface kernels $\theta_1:\Gamma \to \Zp$ and $\theta_2:\Gamma \to \Zp$ that are topologically equivalent relative to the $k$ marked elliptic generators by $\varphi:\Gamma \to \Gamma$, then $\varphi(x_i)$ is conjugated to $x_i$ or $x_i^{-1}$ for all the marked generators $i=1,\ldots, k$, according to the previous comment.} 
\end{rmk}


\subsection{Surface kernels and mapping class groups}

In this part, we define a surface kernel $\theta_G: \Gamma \to \Zp$ for each subgroup of order $p$ in $\modn$. Next, we show that the notion of topological equivalence relative to the $k$ marked elliptic generators, corresponds to conjugation in the set of subgroups of order $p$ in $\modn$.
We proceed by constructing a surface kernel $\theta_G$, and establish a series of lemmas that verify that the $k$ marked elliptic generators correspond to the marked points on $N_g$. Finally, we prove that the notion of topological equivalence introduced in Definition \ref{Def:SurfaceKernel:Marked} is similar to conjugacy classes in this setting.



Let $G\leqslant \modn$ be a subgroup of order $p$. By Nielsen realization Theorem, there exists a group $\w{G}\leqslant \pdiffnk$ and a dianalytic structure $\X\in \M(N_g)$ such that $\pi(\w{G})=G$ and $\w{G}\leqslant \Aut(N_g;\X)$, where $\pi: \pdiff(N_g;k)\to \modn$ is the natural projection. Thus, there exists $f\in \Aut(N_g;\X)$ such that $\w{G}=\langle f \rangle$.
Note that $\w{G}$ acts on the surface $N_g$, and the set of singular points of this action equals the set of fixed points of $f$. Denote by $\{z_1,\ldots,z_k,\ldots ,z_t \}$ the fixed points of $f$ with the convention that the first $k$ points are marked points of the surface $N_g$.

Let $K_G$ be a surface NEC group such that the Klein surface $(N_g; \X)$ is isomorphic to the Klein surface $\hiper/K_G$, where $\hiper/K_G$ is endowed with the dianalytic structure induced by the surface NEC group $K_G$. Specifically, there exists a dianalytic homeomorphism 
\[
    w: (N_g; \X) \to \hiper/ K_G.
\]
Thus, the Klein surface $(N_g;\X)$ can be regarded as the quotient $\hiper/ K_G$ with set of marked points equal to $\{ w(z_1), \ldots, w(z_k) \}$. Additionally, there exists a bijection 
\begin{equation}\label{Eqn:AutNg:AutH2/K}
    \epsilon: \Aut(N_g;\X) \to \Aut(\hiper / K_G)
\end{equation}
defined by $y\mapsto w\circ y \circ w^{-1}$, and denote ${\wk{G}}:=\epsilon(\w{G})\leqslant \Aut(\hiper / K_G)$. Let $q_{G}:\hiper\to \hiper/K_G$ be the quotient map and consider the following NEC group:
\begin{equation}\label{Eqn:NECGroup:GammaG}
    \Gamma_G:=\{ \gamma \in \Isohiper \mid q_{G} \circ \gamma= (\epsilon(f^m)) \circ q_{G} \text{ for some } 0\leqslant m < p \} .    
\end{equation}

Note that $K_G\unlhd \Gamma_G$, hence for each $\gamma\in \Gamma_G$, we obtain an induced homeomorphism $\theta(\gamma): \hiper / K_G \to \hiper / K_G $ defined by $\zeta\cdot K_G \mapsto \gamma(\zeta) \cdot K_G$ for all $\zeta\in \hiper$. Moreover, by definition of $\Gamma_G$, it follows that $\theta(\gamma)\in \wk{G}$. Thus, we have defined an epimorphism
    \begin{equation}\label{Eqn:SurfaceKernel}
        \theta: \Gamma_G \to \wk{G} 
    \end{equation}
with $\ker(\theta )= K_G$, that is, $\theta$ is a surface kernel epimorphism. Recall that the group $\w{G}=\langle f \rangle$ acts freely on $N_g$ except for the $t$ fixed points of $f$, and the orbit space $ N_g/\w{G} \cong \hiper / \Gamma_G $ is homeomorphic to a non-orientable surface $N_h$ of genus $h\geqslant 1$ such that $h$ satisfies the Riemann-Hurwitz equation:
\begin{equation}\label{Eqn:Riemann-Hurwitz}
    g-2=p(h-2)+t(p-1).    
\end{equation}
Hence $\sig(\Gamma_G)=(h;-;[(p)^t])$ and the branched covering $\hiper/K_G\to \hiper/\Gamma_G$ has $t$ branched points, namely $w(z_1), \ldots, w(z_k), \ldots , w(z_t)$, where the first $k$ points are the marked points on $\hiper/K_G$. Note that the covering transformations are given by $\wk{G}=\langle w\circ f \circ w^{-1} \rangle$. 
%
%
Since groups with signature equal to $(h;-;[(p)^t])$ are isomorphic to $\Gamma$, there exists an isomorphism $ \lambda : \Gamma \to \Gamma_G $. 
In fact, in Lemma \ref{Lem:Isomorphism_one_marked_point}, we show that the isomorphism can be always chosen in a way that the first of the generators $x_1, \ldots, x_k$ contains the information of the marked points. To start with this procedure, we make the following remark. 

\begin{rmk}
For any isomorphism $\lambda: \Gamma\to \Gamma_G$, denote by $\{ \zeta_1, \ldots, \zeta_t \}$ the fixed points of the elliptic generators $ \lambda(x_1) , \ldots , \lambda(x_t) $. Note that the classes $\zeta_1\cdot K_G, \ldots, \zeta_t\cdot K_G $ are the branched points of the covering $\hiper/ K_G \to \hiper/ \Gamma_G$ which in turns equal to the fixed points of $w\circ f\circ w^{-1}$ (i.e. $\{w(z_1),\ldots, w(z_t) \}$). Then, the elliptic generators are in one to one correspondence with the set of fixed points of $f$. 
\end{rmk}


Now, we will prove two lemmas with the aim of defining an appropriate surface kernel $\theta_G:\Gamma\to \Zp$, such that it includes the information of the marked points in the first elliptic generators of the NEC group $\Gamma$. 

\begin{lem}\label{Lem:Fixed_points_same_orbit}
    Let $\lambda: \Gamma \to \Gamma_G$ be an isomorphism, $ \zeta_i\in \hiper$ be the fixed point of the elliptic generator $ \lambda(x_i) $ and $\zeta_i\cdot K_G=w(z_j) $ for some $1\leqslant j \leqslant t$. Then, for any $\zeta\in \zeta_i\cdot \Gamma_G$ we have that $\zeta\cdot K_G=w(z_j) $.
\end{lem}
\begin{proof}
    This proof is straightforward since the orbits of the fixed points $\zeta_i \cdot \Gamma_G $ are the branched points of the covering $\hiper/K_G \to \hiper / \Gamma_G$, and this point only have one preimage which is equal to $w(z_j)$.
\end{proof}



\begin{lem}\label{Lem:Isomorphism_one_marked_point} 
    Let $\Gamma_G$ be a NEC group as in the previous description and $w:N_g\to \hiper/K_{G}$ be a dianalytic homeomorphism. Then, there exist an isomorphism $\lambda: \Gamma \to \Gamma_G$ such that the class $\zeta_i\cdot K_G$ of the fixed point $\zeta_i\in \hiper$ of the elliptic generator $\lambda(x_i)$ is equal to $w(z_i)$ for each $i=1,\ldots, k$.
\end{lem}
\begin{proof}
    Since the signature $\sig(\Gamma_G)=(h;-;[(p)^t])$, it follows by \cite{Wilkie66}*{Theorem 4} that there exist an isomorphism $\rho: \Gamma \to \Gamma_G$. We proceed by induction on $k$.
    Case $k=1$. If $\rho$ satisfies the condition of the lemma, let $\lambda=\rho$, and the proof is complete. Otherwise, there exists $m\in \{2,\ldots , t\}$ such that the class $\zeta_1\cdot K_G$ of the fixed point $\zeta_1\in \hiper$ of $\rho(x_m)$ is equal to $w(z_1)$. Define the isomorphism $\varphi: \Gamma \to \Gamma$ as follows:
    \begin{align*}
        \varphi(x_i) &  = \begin{cases}
            (x_1\cdot \ldots \cdot x_{m-1}) \cdot x_m \cdot (x_1\cdot \ldots \cdot x_{m-1})^{-1} & \text{if } i=1, \\
            x_1\cdot x_i \cdot {x_1}^{-1} & \text{if } 2\leqslant i < m,  \\
            x_1 & \text{if } i=m, \\
            x_i & \text{if } m<i\leqslant t,
        \end{cases} 
    \end{align*}
    and $\varphi(d_j) =  d_j \text{ for } j=1,\ldots, h,$ and let be $ \lambda:=\rho\circ \varphi :\Gamma \to \Gamma_G$. By construction $\lambda(x_1)=\rho(\varphi(x_1)) = \rho(x_1\cdot \ldots \cdot x_{m-1})\cdot \rho(x_m) \cdot (\rho  (x_1\cdot \ldots \cdot x_{m-1}))^{-1}$. Thus, if $u:= \rho  (x_1\cdot \ldots \cdot x_{m-1})\in \Gamma_G$, then $u(\zeta_1)\in \hiper $ is the fixed point of $\lambda(x_1)$. By Lemma \ref{Lem:Fixed_points_same_orbit},  $u(\zeta_1)\cdot K_G = \zeta_1\cdot K_G = w(z_1)$, satisfying the condition of the lemma.

    The proof of the case $k=1$ can be adapted to the general case of $k$, defining a variation of $\varphi$ and using that we have constructed an isomorphism for $k-1$ marked points. Thus by induction the lemma holds.
\end{proof}


We are in condition to define a surface kernel epimorphism $\theta:\Gamma \to \Zp$ for each subgroup $G$ of order $p$ of $\modn$ which acts on the surface $N_g$ and the orbit space of this action is a non-orientable surface of genus $h$ with $t$ conic points.

\begin{cons}[Surface kernel] \label{C:SurfaceKernel}
Let $G \leqslant \modn$ a subgroup of order $p$. Then by Theorem \ref{Thm:Main:Torsion} there exist $h,t\geqslant 1$ such that satisfies the Riemann-Hurwitz equation. We define a surface kernel $\theta_G:\Gamma \to \Zp$ in $\SKp$ as the composition of the following maps:
\begin{equation}\label{Eqn:Surface_Kernel_Def}
    \xymatrix {
    \underset{\leqslant\Aut(\hiper)}{\Gamma_G \ar[r]^-{\theta}}       &  \underset{\leqslant\Aut(\hiper/K_G)}{\wk{G}} \ar[r]^-{\varepsilon}  & \underset{\leqslant\diffnk}{\w{G}} \ar[r]^-{\pi}  &  \underset{\leqslant\modnk}{G} \ar[d]^-{\rho}  \\
    \Gamma  \ar[rrr]^-{\theta_G} \ar[u]^-{\lambda}    &                   &           &   \Zp 
}
\end{equation}
where $\rho$ is any isomorphism from $G$ to $\Zp$, $\pi:\phome(N_g;k)\to \modn$ is the natural projection, $\varepsilon$ and  $\theta$ are defined in Equations \eqref{Eqn:AutNg:AutH2/K} and \eqref{Eqn:SurfaceKernel} respectively,   and $\lambda: \Gamma \to \Gamma_G$ is the isomorphism of Lemma \ref{Lem:Isomorphism_one_marked_point}. Notice that $\theta_G$ has the special condition that the isomorphism $\lambda:\Gamma\to \Gamma_G$ contains the information of the marked point. 
\end{cons}




Using the previous construction, we establish a connection between surface kernels up to topologically equivalence relative to $k$ marked generators and conjugacy classes of subgroups of order $p$ of $\modn$. Let $G_1,G_2\leqslant \modn$ be subgroups of order $p$. We denote by $\Gamma_\ell=\Gamma_{G_\ell}$ for $\ell=1,2$ the NEC groups defined in Equation \eqref{Eqn:NECGroup:GammaG} and $\lambda_\ell: \Gamma\to \Gamma_{G_\ell}$ the isomorphism of Lemma \ref{Lem:Isomorphism_one_marked_point}. With this notation the following result holds.

\begin{lem}\label{Lem:Injectivity_of_PSI}
Let $\varphi: \Gamma_1\to \Gamma_2$ be an isomorphism such that $\varphi(\lambda_1(x_i))$ is conjugated to a power of $ \lambda_2(x_i) $ for $i=1,\ldots, k$. Then $G_1$ and $G_2$ are conjugated in $\modn$.
\end{lem}
\begin{proof}
    By \cite{Macbeath67}*{Theorem 3}, the isomorphism $\varphi: \Gamma_{1}\to \Gamma_{2} $ can be realized geometrically, that is, there exists an homeomorphism $ \tau: \hiper \to \hiper $ such that $ \varphi(\gamma)= \tau\cdot \gamma \cdot \tau^{-1} $ for all $\gamma \in \Gamma_{1}$. Since $\ker(\theta_{1})\unlhd \Gamma $, we have that $ \varphi|_{\ker(\theta_{1})} : \ker(\theta_{1}) \to \ker(\theta_{2}) $ is also an isomorphism and is realized geometrically by $\tau$. 
    Thus, there exist $ h : \hiper/ K_1 \to \hiper / K_2 $ and $\widehat{h}: \hiper / \Gamma_1 \to \hiper / \Gamma_2 $ projections of $\tau:\hiper\to\hiper$ to the coverings $q_{K_\ell}: \hiper \to \hiper/K_\ell $ and $q_{\Gamma_\ell}: \hiper \to \hiper/\Gamma_\ell $ respectively. 
    Let $i\in \{ 1,\ldots, k\}$. Observe that 
    \[
        \varphi(\lambda_1(x_i)) = \tau\cdot \lambda_1(x_i) \cdot \tau^{-1}.
    \] 
    On the other hand, there exist $u_i\in \Gamma_2$ and $1 \leqslant m_i < p$ such that $\varphi(\lambda_1(x_i))=u_i \cdot \lambda_2(x_i^{m_i}) \cdot {u_i}^{-1}$. Thus 
    \[
        \tau\cdot \lambda_1(x_i) \cdot \tau^{-1} = u_i \cdot \lambda_2(x_i^{m_i}) \cdot {u_i}^{-1}. 
    \]
    Let $\zeta_i\in \hiper$ be the fixed point of $\lambda_1(x_i)$. From the equation above, it follows that ${u_i}^{-1}(\tau(\zeta_i))$ is the fixed point of $ \lambda_2(x_i^{m_i}) $. By Lemma \ref{Lem:Isomorphism_one_marked_point} we have that $w_2(z_i) = {u_i}^{-1}(\tau(\zeta_i))\cdot K_2$. Finally, applying Lemma \ref{Lem:Fixed_points_same_orbit} and definition of $h:\hiper/K_1\to \hiper/ K_2$ we conclude that 
    \[ 
    w_2 (z_i) = {u_i}^{-1}(\tau(\zeta_i))\cdot K_2 = \tau(\zeta_i)\cdot K_2 = h(\zeta_i\cdot K_2)=h(w_1(z_i)).
    \]
    Since this holds for each $i\in \{1,\ldots,k\}$, it follows that $ \varepsilon(h)= w_2^{-1}\circ h \circ w_1 \in \diffnk $. Therefore $[\varepsilon(h)]\in \modn$ and 
    $
        [\varepsilon(h)] G_1 [\varepsilon(h)^{-1}] = G_2.
    $
\end{proof}



Now, we prove a result that will be useful and whose reference is non-standard.

\begin{lem} \label{Thm:TeichExtremalProblem}
    Let be $\alpha \in \modn$ and $\X_1,\X_2 $ two dianalytic structures in $N_g$. Then there exists $f:(N_g,\X_1)\to (N_g,\X_2) $ such that $f\in \alpha$ and satisfies the following condition:

    \begin{enumerate}
        \item \label{Con:one} Given $ h_i \in \Aut(N_g;\X_i)\cap \diff(N_g;k)$ for $i=1,2$ such that $h_1 \circ f \circ h_2 \simeq f$ relative to the marked points, then $h_1 \circ f \circ h_2 = f $.
    \end{enumerate}
\end{lem}
\begin{proof}
    Let $\phi: \modn \to \mcg(S_{g-1};2k)$ be the homomorphism induced by the orientable double cover $\pi: S_{g-1} \to N_g$ defined by $[f] \mapsto [\w{f}]$ as in \cite{CJX24Per}*{Proposition 2.2} where $\w{f}$ is a lifting under $\pi$ that preserves the orientation of $S_{g-1}$. Consider the pullback structures $\Y_i = \pi^*(\X_i) \in \M(S_{g-1})$. By the Teichmüller Extremal Problem for orientable surfaces with marked points (see \cite{Hubb06}*{Theorem 5.3.12 and Corollary 7.2.3}), there exists a Teichmüller map $\w{f} : (S_{g-1}; \Y_1) \to (S_{g-1}; \Y_2)$ such that $\w{f} \in \phi(\alpha)$. 
    
    Since the structures $\Y_i$ are pullbacks of the dianalytic structures $\X_i$ under $\pi$, the map $\sigma: (S_{g-1}; \Y_i) \to (S_{g-1}; \Y_i)$ is anti-analytic for $i = 1, 2$. Consequently, the composition $\sigma \circ \w{f} \circ \sigma$ has the same complex dilatation as $\w{f}$. Since $\w{f} \in \phi(\alpha)$, we have $\sigma \circ \w{f} \circ \sigma \simeq f$. By the uniqueness of Teichmüller maps \cite{Hubb06}*{Theorem 5.3.12}, it follows that $\sigma \circ \w{f} \circ \sigma = \w{f}$. Then, $\w{f}$ must preserve the fibers of the orientable double cover $\pi: S_{g-1} \to N_g$.

    Define the map $f: (N_g; \X_1) \to (N_g; \X_2)$ as the projection of $\w{f}$ under $\pi: S_{g-1} \to N_g$. Since $\w{f} \in \phi(\alpha)$, it follows that $f \in \alpha$. To prove that $f: (N_g; \X_1) \to (N_g; \X_2)$ satisfies the conditions of the lemma, we consider the lifts of $h_i \in \Aut(N_g; \X_i)$ for $i = 1, 2$, each of which has complex dilatation equal to one. By applying the uniqueness of Teichmüller maps on $S_{g-1}$ and projecting to $N_g$, the result follows, completing the proof.
\end{proof}

The following result, establish that two subgroups of order $p$ that are conjugated in $\modn$, induce surface kernels topologically equivalent relative to the $k$ elliptic generators. Thus, this two notions are equivalent as we will prove in Theorem \ref{Thm:Conjugacy:Classes:Surface:Kernels}.

\begin{lem}
    Let $G_1,G_2\leqslant \modn$ be subgroups of order $p$ such that $G_1$ and $G_2$ are conjugated in $\modn$. 
    Then $\theta_{G_1}$ and $\theta_{G_2}$ are topologically equivalent relative to the $k$ marked elliptic generators.
\end{lem}
\begin{proof}
    Let $\X_1, \X_2$ dianalytic structures of $N_g$ and $\w{G}_1, \w{G}_2 \leqslant \diffnk$ such that $\pi(\w{G}_\ell)=G_\ell$, $ \w{G}_\ell \leqslant \Aut(N_g;\X_\ell) $, for $\ell=1,2$. 
    Let $ K_\ell$ denote the surface NEC group such that $\hiper/K_\ell\cong N_g$ as Klein surfaces with structure $\X_\ell$. Additionally, denote $\wk{G}_\ell = \varepsilon(\w{G}_\ell) \leqslant \Aut(\hiper/K_\ell) $ as in Construction \ref{C:SurfaceKernel}. 

    Since \( G_1 \) and \( G_2 \) are conjugate in \( \modn \), there exists \( \alpha \in \modn \) such that \( \alpha G_1 \alpha^{-1} = G_2 \). By Lemma \ref{Thm:TeichExtremalProblem}, there exists a diffeomorphism \( f \in \diffnk \) such that \( f \in \alpha \) and \( f: (N_g; \X_1) \to (N_g; \X_2) \) satisfies condition \eqref{Con:one} described in the lemma.

    \begin{claim}[1]
        $ f \w{G}_1 f^{-1} = \w{G}_2 $.
    \end{claim}
    Let $y_1\in \w{G}_1$. Since $\alpha {G_1}\alpha^{-1}=G_2$, it follows that  there exists $y_2 \in \w{G}_2$ such that $ f\circ y_1 \circ f^{-1} \simeq y_2  $ relative to the marked points. This implies that $y_2^{-1} \circ f \circ y_1 \simeq f$. Since $y_\ell \in \Aut(N_g;\X_\ell)$ for $\ell=1,2$, it follows that $y_2^{-1} \circ f \circ y_1 = f $ by condition \eqref{Con:one} of Lemma \ref{Thm:TeichExtremalProblem}. This proves that $ f \w{G}_1 f^{-1} \subset \w{G}_2 $. The other contention it can be proved in a similar way. \\

    \n
    Now, define $\overline{f}: = w_2 \circ f \circ w_1^{-1} : \hiper / K_1 \to \hiper / K_2 $ and consider a lifting $\tau: \hiper \to \hiper$ of this homeomorphism. Since the group of covering transformations of $ \hiper / K_\ell \to \hiper / \Gamma_\ell $ is equal to $ \wk{G}_\ell \leqslant \Aut(\hiper/ K_\ell )  $ for each $\ell=1,2$ and since $ f \w{G}_1 f^{-1} = \w{G}_2 $ by Claim (1), then there exists $\widehat{f}: \hiper/\Gamma_1 \to \hiper / \Gamma_2$ such that the following diagram is commutative
    $$ 
    \cuadro{\hiper/K_1}{\overline{f}}{\hiper/K_2}{}{}{\hiper/\Gamma_1}{\widehat{f}}{\hiper/\Gamma_2.} 
    $$
    On the other hand, given $\gamma_1\in \Gamma_1$, the following diagram is commutative:
    $$
    \xymatrix @C=1cm @M=2mm{
	\hiper \ar[r]^-{\tau^{-1}}  \ar[d] & \hiper \ar[r]^-{\gamma_1} \ar[d] & \hiper \ar[r]^-{\tau}  \ar[d] & \hiper \ar[d] \\
	\hiper/ K_2 \ar[r]^-{\overline{f}^{-1}} \ar[d] & \hiper/ K_1 \ar[r]^-{\theta_1(\gamma_1)} \ar[d] & \hiper/ K_1 \ar[r]^-{\overline{f}} \ar[d] & \hiper /K_2  \ar[d] \\
	\hiper/ \Gamma_2 \ar[r]^-{\widehat{f}^{-1}} & \hiper / \Gamma_1  \ar[r]^-{id}	& \hiper / \Gamma_1  \ar[r]^-{\widehat{f}} & \hiper / \Gamma_2,   
	}
    $$
    thus $ \tau \circ \gamma_1 \circ \tau^{-1} \in \Gamma_2$, since this composition is a covering transformation to the branched covering $ \hiper \to \hiper /  \Gamma_2 $. With this information, we define the functions  
    \begin{align*}
        \w{\varphi}: & \Gamma_1 \to \Gamma_2, &  \gamma_1 \mapsto & \tau \circ \gamma_1 \circ  \tau^{-1} \text{ for all } \gamma_1\in \Gamma_1 \\
        \w{\psi}: & \wk{G}_1\to \wk{G}_2, & y_1 \mapsto & f\circ y_1 \circ f^{-1} \text{ for all } y_1 \in \wk{G}_1.
    \end{align*}
    We can see that these functions are isomorphism and the following diagram is commutative
    $$ 
    \cuadro{\Gamma_1}{\theta_1}{\wk{G}_1}{\w{\varphi}}{\w{\psi}}{\Gamma_2}{\theta_2}{\wk{G}_2} 
    $$
%
    This implies that $\theta_1$ and $\theta_2$ are topologically equivalent. Thus, to complete the proof, we need to prove that, in fact, $\w{\varphi}$ preserves the conjugacy classes up to a power of the first $k$ elliptic generators.
    Let $\zeta_i\in \hiper$ be the fixed points of the elliptic generators $\lambda_1(x_i)$ and $\zeta_i'\in \hiper$ be the fixed points of the elliptic generators $\lambda_2(x_i)$ for $i=1,\ldots, k$. Recall from Lemma \ref{Lem:Isomorphism_one_marked_point} that 
    $\zeta_i\cdot K_1=w_1(z_i)$ and $\zeta_i'\cdot K_2=w_2(z_i)$ for $i=1,\ldots, k$. Starting from this fact and since $\tau$ is a lift of $\overline{f}$, we have that
    \begin{align*}
        \zeta_i' \cdot K_2=  w_2 (z_i) =\overline{f}(w_1 (z_i) )= \overline{f}(\zeta_i\cdot K_1) = \tau(\zeta_i)\cdot K_2,  
    \end{align*}
    thus there exists $u_i\in K_2$ such that $\zeta_i'= u_i ( \tau( \zeta_i ) ) $, for each $i=1,\ldots, k$.
    On the other hand, since 
    $$
    \w{\varphi} (\lambda_1(x_i))=\tau \cdot \lambda_1(x_i) \cdot \tau^{-1}, 
    $$  
    it follows that $u_i\cdot ( \w{\varphi}(\lambda_1(x_i)) )\cdot u_i^{-1}$ has as fixed point to $ u_i(\tau(\zeta_i)) $. Then the element $u_i\cdot ( \w{\varphi}(\lambda_1(x_i)) )\cdot u_i^{-1}$ is a power  of $ \lambda_2(x_i) $ because both elements have the same fixed point and order (see \cite{Macbeath67}*{Page 1198}). Consequently, the isomorphism $ {\varphi}: \Gamma \to \Gamma $ defined by $ \lambda_2^{-1}\circ \w{\varphi} \circ \lambda_1  $ satisfies that $\varphi(x_i)$ is conjugated to a power of $x_i$ for $i=1,\ldots, k$. Then $\theta_{1}$ and $\theta_{2}$ are topologically equivalent relative to the $k$ marked elliptic generators. Therefore $\theta_{G_1}$ and $\theta_{G_2}$ are topologically equivalent relative to the $k$ marked elliptic generators.
\end{proof}


Before proving the main theorem of this subsection, we introduce the following notation. Let $\CClk$ be the group of conjugacy classes of subgroups of order $p$ in $\modn$ that act on $N_g$ with $t$ fixed points, such that the orbit space of this action is homeomorphic to $N_h$. Using this notation, Theorem \ref{Thm:Main:CClass:SK} can be restated as follows.

\begin{thm}\label{Thm:Conjugacy:Classes:Surface:Kernels}
    The function $\Psi: \CClk \to \SKtopk  $ defined by $[G]\mapsto [\theta_G]$ is a bijection.
    
\end{thm}
\begin{proof} 
The injectivity of $\Psi$ is consequence of Lemma \ref{Lem:Injectivity_of_PSI}. It remains to prove the surjectivity. Let $\theta:\Gamma \to \Zp$ be a surface kernel such that $[\theta]\in \SKtopk$. Consider a NEC group $\Gamma'$ isomorphic to $\Gamma$ and a fixed isomorphism $\lambda: \Gamma \to \Gamma'$, then $ \theta'=\theta\circ \lambda^{-1}: \Gamma' \to \Zp  $ is also a surface kernel which is equivalent to $\theta$. Let be $K:=\ker(\theta')\lhd \Gamma'$ which is a surface NEC group. Since 
$ 
\Gamma'/K \cong \Zp 
$ 
then by \cite{Bujalance10}*{Theorem 1.4.4} the group $\Zp$ is a subgroup of $\Aut(\hiper/K)$. More precisely, $\Zp\cong \cov(q)$ where $\cov(q)$ are the covering transformations of the branched covering $\hiper/K\to \hiper/\Gamma'$. In particular, the surface kernel $\theta': \Gamma'\to \cov(q)$ can be thought as follows. For each $\gamma\in \Gamma'$, $\theta'(\gamma):\hiper/K\to  \hiper/K$ is the covering transformation defined by $\theta'(\gamma)(\zeta\cdot K)=\gamma(\zeta)\cdot K$ for all $\zeta\in \hiper$.

Define $\wk{G}:=\cov(q)\leqslant \Aut(\hiper/K)$, and note that $\wk{G}=\langle f \rangle$ for some $f\in \cov(q)$. It remains to prove that $f$ fix the marked points. For each $i\in \{1,\ldots, k\}$ let $\zeta_i\in \hiper$ be the fixed point of $\lambda(x_i)$. Then $ \theta'(\lambda(x_i))(\zeta_i\cdot K) = \zeta_i\cdot K $ and also $\theta'(\lambda(x_i))\neq id_{\hiper/K}$, since $\lambda(x_i)$ has finite order. Since $\theta'(\lambda(x_i))$ is a power of $f$ and $p$ is an odd prime, it follows that $f$ preserves the classes $\zeta_i\cdot K$ for all $i=1,\ldots, k,$. 

Now, consider a dianalytic homeomorphism $w: N_g\to \hiper/K$ such that $w(z_i)=\zeta_i\cdot K$. Hence $\w{G}=\langle w^{-1}\circ f \circ w \rangle \leqslant \diffnk $. By the arguments in the previous paragraphs, the group $G=\pi(\w{G})$ satisfies that the surface kernel $\theta_G$ is topologically equivalent relative to the marked point to $\theta:\Gamma \to \Zp$.
\end{proof}

\subsection{Surface kernels and t-tuples}
In this section, we establish a connection between the set  of the surface kernels $\SKtopk$ and the set $\Ttuples$ of equivalence classes of $t$-tuples of $(\Zp)^t$ for a solution $(h,t)$ of the Riemann-Hurwitz equation \eqref{Eqn:Riemann-Hurwitz}. This approach reduces the problem of finding conjugation classes of subgroups of order $p$ to a purely combinatorial problem. More precisely, the problem is reduced to finding all the possible equivalence classes of $t$-tuples that are admissible for the solution $(h,t)$. 

We begin by introducing a series of automorphisms of the group $\Gamma$ and results related to this automorphism. Consider the automorphisms of $\Gamma$ described in \cite{Bujalance15}, as presented in the table below (at the generator level). By convention, generators that remain fixed are omitted.

\begin{center}
        \begin{tabular}{|c r l l|}
\hline 
$\gamma:$ & $\Gamma$ & $\to$ & $\Gamma$   \\ 
\hline 
 
 & $d_1$ & $\mapsto$ & $x_t\cdot d_1$   \\ 
 & $x_t$ & $\mapsto$ & $(x_t \cdot d_1)\cdot {x_t}^{-1} \cdot (x_t \cdot d_1)^{-1}$  \\ 

\hline 
$\rho_i:$ & $\Gamma$ & $\to$ & $\Gamma$   \\ 
\hline 
 
 & $x_i$ & $\mapsto$ & $x_i \cdot x_{i+1} \cdot  {x_i}^{-1}$  \\ 
 & $x_{i+1}$ & $\mapsto$ & $x_i$   \\ 

\hline 
$\alpha_j:$ & $\Gamma$ & $\to$ & $\Gamma$ \\
\hline 
 & $d_j$ & $\mapsto$ & ${d_j}^2 \cdot d_{j+1} \cdot  {d_j}^{-2}$  \\ 
 & $d_{j+1}$ & $\mapsto$ & $d_j$   \\ 

\hline 
    \end{tabular} 
        \end{center}
Where $i=1,\ldots t-1$ and $j=1,\ldots, h-1$. 

For the remainder of the section and in order to simplify the exposition we add the following convention. When we refer to applying a sequence of automorphisms $\varphi_n, \dots, \varphi_1$, where each $\varphi_\ell$ represents an isomorphism as outlined in the table, this means that we first apply $\varphi_1$, transforming the generators into $\varphi_1(x_i)$ and $\varphi_1(d_j)$. We then apply $\varphi_2$ to these newly defined generators, in which $\varphi_1(x_i)$ and $\varphi_1(d_j)$ take the roles of the $x_i$'s and $d_j$'s, as indicated in the table. This process continues until the final automorphism $\varphi_n$. To clarify this convention, we present the following example.

\begin{example}
We construct the automorphism $\varphi$ using $ \gamma \alpha_1$:
\begin{align*}
        & \alpha_1 &  & \gamma   &  \\
    x_t & \mapsto x_t & ={x_t}' & \mapsto  ({x_t}'{d_1}')\cdot ({x_t}')^{-1} \cdot ({x_t}'{d_1}')^{-1} & ={x_t}'' \\
    d_1 & \mapsto {d_1}^2\cdot d_2 \cdot {d_1}^{-2} & ={d_1}' & \mapsto  {x_t}'{d_1}' & ={d_1}'' \\
    d_2 & \mapsto d_1 & {=d_2}' & \mapsto {d_2}' & ={d_2}''
\end{align*}
Finally, we can express the last generators in terms of the original as we expect and this directly gives us the isomorphism $\varphi$ using ${\gamma \alpha_1}$:
\begin{align*}
    \varphi(x_t) & ={x_t}''=  (x_t \cdot {d_1}^2\cdot d_2 \cdot {d_1}^{-2} )\cdot {x_t}^{-1} \cdot (x_t \cdot {d_1}^2\cdot d_2 \cdot {d_1}^{-2} )^{-1}\\
    \varphi(d_1) & ={d_1}''=  x_t \cdot {d_1}^2\cdot d_2 \cdot {d_1}^{-2} \\
    \varphi(d_2) & ={d_2}''=  d_1\\
\end{align*}
\end{example}
As showed in the previous example, it is difficult to precisely describe the isomorphism $\gamma \alpha_1$ using the canonical generators $x_1, \ldots, x_t, d_1 \ldots  d_h$; doing so would require significant time and text. Because of this, it is important to keep in mind that when using the automorphisms described in the table, it means a sequence of generator changes was made and that the isomorphism is not exactly the result of the composition. The important point to keep in mind is the effect that these isomorphisms have on the $k$ marked elliptic generators. When we make the pertinent change of generator we must take into consideration the effect of the isomorphism has into the conjugacy class of this $k$-elliptic generators. At the end, the canonical elliptic generators form a complete elliptic system in the sense of \cite{Bujalance90}*{Definition 2 of Section 2.2}.


\begin{lem} \label{Lem:Change_of_epimorphism_dj_equal_0}
Let $\theta: \Gamma \to \Zp$ be a surface kernel. Then, there exists an isomorphism $\varphi: \Gamma \to \Gamma$ such that
\begin{align*}
    \varphi(x_i) &= x_i & \text{for } & i=1,\ldots, t-1, \\
    \theta(\varphi(d_j)) &= 0 & \text{for } & j=2,\ldots, h, \\
    2 \cdot \theta(\varphi(d_1)) &= - \sum_{i=1}^{t} \theta(x_i),
\end{align*}
and $\varphi(x_t)$ is in the same conjugacy class as $x_t$. In other words, $\theta$ is topologically equivalent relative to the $k$ marked elliptic generators to $\theta' = \theta \circ \varphi$ and satisfies $\theta'(x_i) = \theta(x_i)$ for $i=1,\ldots, t$, $\theta'(d_j) = 0$ for $j=1,\ldots, h$, and $2 \cdot \theta'(d_1) = -\sum_{i=1}^{t} \theta(x_i). $
\end{lem}

\begin{proof}
This proof is similar to \cite{Bujalance15}*{Theorem 1} but adapted for our purposes. Using isomorphisms of the type $\alpha_j$, we obtain new generators $x_i' = x_i$ for $i=1,\ldots, t$ and $d_j'$ such that $\theta(d_j') = 0$ for $j \geq m$ and $\theta(d_j') \neq 0$ for $j < m$. Assume $m > 2$. Since $\theta(x_r') \neq 0$, there exists $b$ such that $\theta(d_1') + b \cdot \theta(x_r') = 0$. Thus, using the automorphism $(\gamma \alpha_1 \gamma \alpha_1)^b$, we obtain new generators $x_i'' = x_i$ for $i=1,\ldots, t-1$, with $x_t''$ in the same conjugacy class as $x_t$, and $\theta(d_j'') = 0$ for $j \geq m$ and $\theta(d_1'') = 0$. Reordering the generators $d_j''$, using the isomorphism $\alpha_j$, we can assume that $\theta(d_j'') = 1$ for $j \geq m-1$. Repeating this process, we have the final basis where $x_i''' = x_i$ for $i=1,\ldots, t-1$, with $x_t'''$ in the same conjugacy class as $x_t$, $\theta(d_j''') = 0$ for $j > 1$, and $\theta(d_1''')$ is given by
\[
2 \cdot \theta(d_1''') = - \sum_{i=1}^{t} \theta(x_i).
\]
Expressing the basis $x_i''', d_j'''$ in terms of the original basis, we obtain the desired isomorphism $\varphi: \Gamma \to \Gamma$.
\end{proof}

The previous result, shows that for each $[\theta]\in \SKtopk$ there exists a representative $\theta:\Gamma\to \Zp$ such that $\theta(d_j)=0$ for all $j=2,\ldots, h .$ The following result shows that the sign of the image of a surface kernel $\theta:\Gamma\to \Zp$ of the elliptic generators $x_i$ can be changing to our convenience without affect the topologically equivalent class. That is, for a given collection of change of signs $(\varepsilon_1, \ldots, \varepsilon_t )\in \{-1,1 \}^{t}$ there exists $\theta':\Gamma\to \Zp$ topologically equivalent to $\theta$ relative to the $k$ marked elliptic generators such that $\theta'(x_i)=\varepsilon_i \cdot \theta(x_i)$.

\begin{lem}\label{Lem:change_of_sign_xi}
    Let $\theta: \Gamma \to \Zp$ be a surface kernel such that $\theta\in \SKp$ and $\varepsilon_i\in \{-1,1\} $ for $i=1,\ldots, t$. Then there exists an isomorphism $\varphi:\Gamma \to \Gamma$ such that 
    \begin{align*}
        \theta(\varphi(x_i)) &  =\varepsilon_i\cdot \theta(x_i)  & \text{for } & i=1, 2,\ldots, t\\
        2\cdot \theta(\varphi(d_1) )&   = -2\cdot \sum_{j=2}^{h} \theta(d_j)  - \sum_{i=1}^{t} \varepsilon_i \cdot \theta(x_i)  \\
        \varphi(d_j) &          =d_j & \text{ for } & j=2, 3,\ldots, h   
    \end{align*}
    and  if $\varepsilon_i=1$, then $\varphi(x_i)$ belongs to the conjugacy class of $ x_i $, otherwise $\varphi(x_i)$ belongs to the conjugacy class of $x_i^{-1}$.
\end{lem}
\begin{proof}
    We define the functions $\varphi_i: \Gamma \to \Gamma$ for each $i=1,\ldots, t$ at the level of generators as follows:
\begin{align*}
    x_l & \mapsto x_l & \text{for } & l=1,\ldots i-1 \\
    x_i & \mapsto (x_i \cdot d_1)\cdot {x_i}^{-1} \cdot (x_i \cdot d_1)^{-1} & &  \\
    x_l & \mapsto (x_i \cdot d_1 \cdot x_i \cdot {d_1}^{-1}) \cdot x_l \cdot (x_i \cdot d_1 \cdot x_i \cdot {d_1}^{-1})^{-1} & \text{for } & l=i+1,\ldots t \\  
    d_1 & \mapsto x_i\cdot d_1 & & \\
    d_j & \mapsto d_j & \text{for } & j=2,\ldots, h. 
\end{align*}
Since the relation of the group $\Gamma$ is preserved by $\varphi_i$, it follows that $\varphi_i$ defines a group homomorphism from $\Gamma$ to $\Gamma$. Moreover, it can be checked that  $\varphi_i$ is  an isomorphism. Define $a_i:=(\varepsilon_i+1)/2$ for $i=1,\ldots, t$. Using the isomorphisms $\varphi_1^{a_1} \varphi_{2}^{a_2} \ldots \varphi_{t-1}^{a_{t-1}}\varphi_t^{a_t}$ we obtain the desired isomorphism $\varphi$.
\end{proof}

Now, consider the set of $t$-tuples $(\beta_1, \ldots , \beta_t )$ such that $\beta_i\in (\Zp)^{\times}=\Zp\setminus \{ 0 \}$. We introduce the notion of congruence in the set of $t$-tuples.

\begin{defi}
    Let $(h,t)$ be solution to the Riemann-Hurwitz equation \ref{Eqn:Riemann-Hurwitz} and $(\beta_1, \ldots , \beta_t)$ and $(\beta_1', \ldots , \beta_t')$ be two $t$-tuples. We say that $(\beta_1, \ldots, \beta_k \mid \beta_{k+1} , \ldots, \beta_t)$ and $(\beta_1', \ldots , \beta_k' \mid \beta_{k+1}',\ldots \beta_t')$ are \textit{congruent} if there exist a permutation $\sigma:\{k+1,\ldots ,t\}\to \{k+1, \ldots , t\} $, an integer $a\not\equiv 0 \mod{p}$ and $\varepsilon_i\in \{ -1,1 \}$ for $i=1,\ldots, t$ such that 
    $$
        \beta_i = \begin{cases}
            a\cdot \varepsilon_i\cdot \beta_i' & \text{if } i=1,\ldots , k \\
            a\cdot \varepsilon_i\cdot \beta_{\sigma(i)}' & \text{if } i=k+1,\ldots , t.
        \end{cases}
    $$
\end{defi}

We can see that the notion of congruence in the set of $t$-tuples is an equivalence relation and we write $\cong$ for this relation. We denote the set of equivalence class of $t$-tuples as  
$$ \Ttuplesk:= (\Zp\setminus \{ [0] \})^t / \cong $$ 
%

\begin{thm}\label{Thm:FixedPointData}
    There is a one-to-one correspondence $\Phi: \SKtopk \to \Ttuplesk $ defined by $[\theta]\mapsto [(\theta(x_1),\ldots,\theta(x_k) \mid \theta(x_{k+1}),\ldots , \theta(x_t) )]$
\end{thm}
\begin{proof} 
    The function is well defined. Let $\theta_1$ and $\theta_2$ be two surface kernels such that $\theta_1$ and $\theta_2$ are topologically equivalent relative to the $k$ marked elliptic generators. Then there exist $\varphi: \Gamma\to \Gamma$ and $\psi:\Zp\to \Zp$ such that the Diagram \ref{Eqn:Diagram_equivalence_of_SK} is commutative. By Remark \ref{Rmk:conjugated_to_x_or_x-1}, we have that $\varphi(x_i)$ is conjugated to $x_i$ or $x_i^{-1}$ for each $i=1,\ldots, k$, and also there exists a permutation $\sigma:\{k+1,\ldots, t\}\to \{k+1,\ldots, t\}$ such that $ \varphi(x_i)$ is conjugated to $ x_{\sigma(i)}$ or $ x_{\sigma(i)}^{-1}$ for $i=k+1,\ldots, t$. Define 
    \begin{align*}
        \varepsilon_i=\begin{cases}
            1 & \text{if } x_i \text{ is conjugated to } x_i \\
            -1 & \text{if } x_i \text{ is conjugated to } x_i^{-1}
        \end{cases}
    \end{align*}
    for each $i=1,\ldots, k$. In a similar way and using the permutation $\sigma$, we define $\varepsilon_i$ for each $i=k+1,\ldots, t$. Notice that the isomorphism $\psi:\Zp\to \Zp$ is a multiplication for some $a\in \{ 1,\ldots ,p-1\}$. Hence, 
    \[
        \theta_2(\varphi(x_i))=\psi(\theta_1(x_i))=a\cdot \theta_1(x_i),
    \] 
    but $\theta_2(\varphi(x_i))=\varepsilon_i\cdot \theta_2(x_i)$ for $i=1,\ldots, k$, and $\theta_2(\varphi(x_i))=\varepsilon_i \cdot \theta_2(x_{\sigma(i)})$ for $i=k+1,\ldots, t.$ Then, the $t$-tuples $(\theta_1(x_1),\ldots, \theta_1(x_k)\mid \ldots, \theta_1(x_t)))$ and $(\theta_2(x_1),\ldots, \theta_2(x_k)\mid \ldots, \theta_2(x_t)))$ are congruent, and thus $\Phi$ is well defined.

    \n
    Now, we prove that $\Phi$ is injective. Let $\theta_1:\Gamma\to \Zp$ and $\theta_2: \Gamma \to \Zp$ be two surface kernels such that 
    \begin{align*}
        (\theta_1 (x_1), \ldots , \theta_1(x_k) & \mid \theta_1 (x_{k+1}), \ldots,\theta_1(x_t) ) \cong \\
        (\theta_2 &  (x_1),  \ldots , \theta_2(x_{k})\mid \theta_2 (x_{k+1}), \ldots,\theta_2(x_t) ).
    \end{align*}
    By definition of congruence there exists a permutation $ \sigma: \{ k+1,\ldots , t\} \to \{ k+1,\ldots , t\} $, an integer $a\neq 0$, and $\varepsilon_i \in \{-1,1 \}$ for $i=1,\ldots, t$ such that  
    $$
    \theta_1(x_i)= \begin{cases}
        a\cdot \varepsilon_i \cdot \theta_2 (x_i) & \text{if } 1\leqslant i \leqslant k, \\
        a\cdot \varepsilon_i \cdot \theta_2 (x_{\sigma({i})}) & \text{if } k+1\leqslant i \leqslant t,
    \end{cases}
    $$
    By Lemma \ref{Lem:Change_of_epimorphism_dj_equal_0}, for each $\ell=1,2$ there exist a epimorphism $\theta_\ell':\Gamma\to \Zp$ topologically equivalent relative to the $k$ marked elliptic generators to $\theta_\ell:\Gamma\to \Zp$ such that:
    \begin{align*}
        \theta_\ell'(x_i) & =\theta_\ell(x_i) & \text{for } & i=1,\ldots, t,\\
        \theta_\ell'(d_j) & = 0  & \text{for } & j=2,\ldots, h,  \\
        2\cdot \theta_\ell'(d_1) & =- \sum_{i=1}^{t} \theta_\ell(x_i).
    \end{align*}
    Now, using isomorphism of the type $\rho_i$ (for $i\geqslant 2$) we obtain a new set of generators ${x_i}',$ ${d_j}'$ such that ${x_i}'=x_i$ for $1\leqslant i \leqslant k$, $ {d_j}'=d_j $ for $j=1,\ldots, h$ and $\theta_2'({x_i}')=\theta_2'(x_{\sigma(i)})$ for $i>k$. By Lemma \ref{Lem:change_of_sign_xi} there exists an isomorphism $\phi:\Gamma \to \Gamma$ such that $\theta_2'(\phi({x_i}'))=\varepsilon_i \cdot \theta_2'({x_i}') $ for $i=1,\ldots, t$, $\phi({d_j}')={d_j}'$ for $j=2,\ldots, h$, and $ 2 \cdot \theta_2'(\phi(d_1))= -\sum_{i=1}^{t} \varepsilon_i \cdot \theta_2'({x_i}') $. Define the isomorphism $\varphi: \Gamma \to \Gamma$ by sending $x_i \mapsto \phi(x_i')$ and $d_j\mapsto \phi(d_j')$ for $i=1,\ldots, t$ and $j=1,\ldots, h$, and $\psi:\Zp\to\Zp$ defined by multiplication by $a^{-1}\in \Zp$. We can see that $\varphi(x_i)$ is conjugated to $x_i^{\varepsilon_i}$ for $i=1,\ldots, k$. And also
    \begin{align*}
        \theta_2' (\varphi(x_i))  & = \theta_2' (\phi(x_i')) \\
                & = \varepsilon_i \cdot \theta_2' ( x_i' ) \\
                & = \varepsilon_i \cdot \theta_2' ( x_{\sigma(i)} ) \\
                & = \varepsilon_i \cdot \theta_2 ( x_{\sigma(i)} ) 
    \end{align*}    
    for $k+1 \leqslant i \leqslant t$. Thus $\psi^{-1} (\theta_2' (\varphi(x_i))) = a\cdot \varepsilon_i \cdot \theta_2 ( x_{\sigma(i)} ) = \theta_1(x_i)=\theta_1'(x_i)$ for $i=k+1,\ldots, t$. In a similar way we can see that $\psi^{-1} ( \theta_2'(\varphi(x_i)))=\theta_1'(x_i)$ for $i=1,\ldots, k$ and $ \psi^{-1} ( \theta_2'(\varphi(d_j)))=\theta_1'(d_j) $ for $j=1,\ldots, h. $ Hence $\theta_2'$ and $\theta_1'$ are topologically equivalent relative to the $k$ marked elliptic generators. Therefore $[\theta_1]=[\theta_2]$, proving the injectivity. 

    Finally, we prove the surjectivity of $\Phi$. Let $[(\beta_1,\ldots,\beta_k\mid \beta_{k+1},\ldots, \beta_t)]\in \Ttuplesk$. Define $\mu:=1+\sum_{i=2}^{t}\beta_i$,
    \begin{align*}
        \epsilon(\mu):= & \begin{cases}
            1 & \text{if } \mu  \text{ is even} \\
            -1 & \text{if } \mu  \text{ is odd}
        \end{cases} &  
        \delta (\mu):= & \begin{cases}
            -\mu & \text{if } \mu  \text{ is even} \\
            -(p-1)-(\mu-1) & \text{if } \mu  \text{ is odd}
        \end{cases}
    \end{align*}
    and a function $\theta:\Gamma \to \Zp$ given by:
    \begin{align*}
        \theta(x_i) = & \begin{cases}
            \epsilon(\mu) & \text{if } i=1 \\
            \beta_i & \text{if } 2\leqslant i \leqslant t
        \end{cases} & 
        \theta(d_j) = & \begin{cases}
            \delta(\mu)/2 & \text{if } i=1 \\
            0 & \text{if } 2\leqslant i\leqslant h.
        \end{cases}
    \end{align*}
    Since the relation of the group $\Gamma$ are preserved by the function, we have that $\theta$ is actually an homomorphism. Moreover, $\theta$ is an epimorphism because $\theta(x_1)$ is equal to $1$ or $-1$, where both of them are generators of $\Zp$. On the other hand, $\ker(\theta)\lhd \Gamma$ and by the Riemann-Hurwitz equation we can see that $\ker(\theta)\cong \pi_1(N_h)$ of a non-orientable surface of genus $h$. The proof is complete by seeing that $\Phi[\theta]=[(\beta_1,\ldots,\beta_k\mid \beta_{k+1},\ldots, \beta_t)]\in \Ttuplesk$. 
\end{proof}

\begin{rmk}
    Notice that if $k\geqslant 1$, then each equivalence class in $ \Ttuplesk $ has a representative $(1,\beta_2,\ldots, \beta_k\mid \beta_{k+1},\ldots, \beta_t)$.  Thus, we restrict the set of $t$-tuples to those that have the form $(1,\beta_2,\ldots, \beta_k\mid \beta_{k+1},\ldots, \beta_t)$ and in these we define the notion of congruence, i.e. 
    $(1,\beta_2,\ldots, \beta_k\mid \beta_{k+1},\ldots, \beta_t)\cong (1,\beta_2',\ldots, \beta_k'\mid \beta_{k+1}',\ldots, \beta_t')$ 
    if and only if there exist a permutation 
    $\sigma:\{k+1,\ldots ,t \} \to \{ k+1, \ldots ,t\}$ and 
    $\varepsilon_i\in \{-1,1\}$ for $ i=2,\ldots, t$ such that $\beta_i=\varepsilon_1\cdot \beta_i'$ for $i=2,\ldots k$ and $ \beta_i=\varepsilon_1\cdot \beta_{\sigma(i)}'$ for $i=k+1,\ldots, t$. 
\end{rmk}

The previous remark in combination with Theorem \ref{Thm:FixedPointData} and Theorem \ref{Thm:Main:CClass:SK} implies that there is a one-to-one correspondence:
    \[
        \left\lbrace  \begin{array}{c}
		\text{Conjugacy classes of subgroups} \\ 
		\text{of order $p$ in }\mathcal{N}_g^k \text{ that acts on } N_g \\
		\text{ with } t \text{ fixed points }
		\end{array}  \right\rbrace \leftrightarrow
        \left\lbrace  \begin{array}{c}
	\text{Congruence classes of }t \text{-tuples} \\		
		(1, \beta_2 \ldots, \beta_k \mid \beta_{k+1} ,\ldots, \beta_t)  \\ 
		
		\text{with } 0<\beta_j<p 
		\end{array} \right\rbrace           
    \]
which proves Theorem \ref{Thm:Main:CClass:tTupl}. We now apply this result to the case of $\modnpk$, obtaining the following result.

\begin{cor}\label{Cor:Number_of_Conjugacy_clases_Npk}
    The number of conjugacy classes of subgroups of order $p$ in $\mathcal{N}_p^1$ \ and \ $\mathcal{N}_p^2$ is equal to $\tfrac{p-1}{2}.$ 
\end{cor}
\begin{proof}
    By Example \ref{Ex:genus_p} the group $\modnpk$ has a subgroup of order $p$ and the unique solution of the Riemann-Hurwitz equation is $(h,t)=(1,2).$ Then by he previous theorem and remark the number of conjugacy classes is equal to the $2$-tuples $(1 \mid \beta_2)$ under congruence equivalence. Notice that the possible $2$-tuples are 
    $$ 
    (1\mid 1), (1 \mid 2), (1 \mid 3), \ldots (1 \mid p-2), (1 \mid p-1) 
    $$
    and under the congruence relation we have
    $$
    (1 \mid 1 ) \cong (1\mid p-1 ),  \ \ (1 \mid 2) \cong (1 \mid p-2), \ldots, \  \left( 1 \, \left\vert \, \frac{p-1}{2} \right. \right) \cong \left(1 \, \left\vert \, \frac{p+1}{2} \right. \right).
    $$
    Thus, the possible values for $\beta_2$ up to congruence are $1\leq \beta_2 \leq \tfrac{p-1}{2}$, proving the result for $k=1$. The case with two marked points are similar and we omit it. 
\end{proof}


\section{Normalizers of subgroups of order \texorpdfstring{$p$}{p} in \texorpdfstring{$\modn$}{Ngk}}
\label{Sec:Normalizers}

In \cite{CJX24Per}*{Theorem 1}, N. Colin, R. Jiménez-Rolland, and M.A. Xicoténcatl prove that the pure mapping class group $\modn$ has $p$-periodic cohomology whenever $\modn$ contains $p$-torsion. In Theorem \ref{Thm:Main:Torsion}, we prove that $\modn$ has $p$-torsion if there exist solutions to the Riemann-Hurwitz equation. Thus, we fix $g \geqslant 2$, $k \geqslant 1$, and an odd prime number $p$ such that $\modn$ has $p$-torsion (i.e., there exists at least one solution $h \geqslant 1, t \geqslant k$ to the Riemann-Hurwitz equation). By Brown's Theorem, we have that the Farrell cohomology of the pure mapping class group $\modn$ is determined by:
$$
    \Farrp{*}{\modn} \cong \prod_{\Zp \in S} \Farrp{*}{N(\Zp)},
$$
where $S$ is a set of representatives of the conjugacy classes of subgroups of order $p$ in $\modn$, and $N(\Zp)$ is the normalizer of $\Zp$ in $\modn$.

In the previous section, we established a one-to-one correspondence between conjugacy classes of subgroups of order $p$ in $\modn$ and congruence classes of $t$-tuples $(1, \beta_2, \ldots, \beta_k \mid \beta_{k+1}, \ldots, \beta_t)$, reducing the study of conjugation classes into a purely combinatorial problem.

The final ingredient to obtain the $p$-primary component of Farrell cohomology is to develop techniques that allow us to calculate the cohomology of the normalizers $N(\Zp)$ for the $\Zp$ representatives of the conjugacy classes. With this objective in mind, in this section we obtain a group extension:
$$
    1 \to \Zp \to N(\Zp) \to H \to 1,
$$
where $H \leqslant \mcg(N_h; t)$. We start with some notation and results about Birman-Hilden theory.


\subsection{Birman-Hilden theory} 

Let $\Zp\leqslant \modn$ be a subgroup of order $p$. Then by Nielsen realization theorem there exist a subgroup $G\leqslant \pdiffnk$ such that $ \pi(G)\cong \Zp $. Consider the branched covering $q: N_g\to N_g / G\cong N_h$ which has $t$ ramified points. By Birman-Hilden theory in non-orientable surfaces \cite{AM20BH}*{Theorem 1.1} there exist a well-defined homomorphism 
$$ 
    \widehat{I}: \Smodnk \to \Lmodnht \ \ \ [\w{f}]\mapsto [f], 
$$
where $\Lmodnht$ is a finite-index subgroup of $\mcg(N_h;t)$ formed by mapping classes of $N_h$ which has a representative $f\in \diff(N_h;t)$ that lifts to a diffeomorphism of $N_g$ by $q:N_g\to N_h$; the subgroup $\Smodnk$ of $\modnk$ is formed by mapping classes which has a representative $\w{f}\in \diffnk$ that preserves the fibers of $q:N_g\to N_h$. The groups $ \Lmodnht $ and $\Smodnk$ are called \textit{the liftable mapping class group} and \textit{the symmetric mapping class group respectively.}

We restrict the homomorphism $\widehat{I}$ to the pure mapping class group $\modn$ and denote by $\Smodn:=\modn\cap \Smodnk$ the symmetric mapping class group and denote by
$$
I: \Smodn \to \Lmodnht
$$ 
the homomorphism obtained by composing with the restriction homomorphism. Notice that $\ker(I)=\pi(\cov(q))=\pi(G)\cong \Zp$. Now, let $\mcg^k(N_h;t)$ be the subgroup of $\mcg(N_h;t)$ formed by mapping classes that fix pointwise the first $k$-marked points and permute the remaining points. Since the homomorphism $I$ is restricted to mapping classes that fix pointwise the $k$-marked points, then $\Ima(I)\leqslant \mcg^k(N_h;t)$.  Moreover, by properties of coverings spaces we can see that the subgroup of $\pdiffnk$ formed by all the elements that preserves fibers are equal to the normalizer of $G$ in $\pdiffnk$, and this fact is transfered to $\modnk$. We summarize the previous comments in the following proposition without a proof.

\begin{prop}\label{Prop:Extension_of_Normalizers}
    Let $\Zp\leqslant \modn$ and $q:N_g\to N_h$ be the induced branched covering. Let $I:\Smodn \to \mcg(N_h;t)$ be the induced homomorphism by the branched covering $q$ by the action of $\Zp$ on $N_g$. Then $\Smodn = N(\Zp)$, $\ker(I)\cong \Zp$ and $\Ima(I)\leqslant \mcg^k(N_g;t)$, where $N(\Zp)$ is the normalizer of $\Zp$ in $\modn$. In other words, we have a group extension:
    $$
        1\to \Zp \to N(\Zp) \to \Ima(I) \to 1.
    $$
\end{prop}

\subsection{Case \texorpdfstring{$\mathcal{N}_{p}^{k}$}{Npk}, \texorpdfstring{$k=1,2$}{k=1,2}}

We know by Example \ref{Ex:genus_p} that the surface $N_p$ of genus an odd prime $p$, the pure mapping class groups $\mathcal{N}_p^1$ and $\mathcal{N}_p^2$ have $p$-torsion. Furthermore, there is a unique solution to the Riemann-Hurwitz equation whose values are $(h,t)=(1,2)$. Let us consider $\Zp \leqslant \modnpk$ and the diffeomorphism $f\in \pdiff(N_p;k)$ that realizes $\Zp$ for $k=1,2$. Then, via the action of $f$ on the surface $N_p$ we obtain a branched covering
\[
q:N_p \to N_1 = \R P^2 
\]
with two branch points. By Proposition \ref{Prop:Extension_of_Normalizers} we have that $q:N_p \to \R P^2 $ induces an extension of groups
\[
        1\to \Zp \to N(\Zp) \to \Ima(I) \to 1,
\]
where $\Ima(I)\leqslant \mcg^k(N_1;2)$. Since $k=1,2$, then $\mcg^k(N_1;2)=\mathcal{N}_1^2$. On the other hand, denote by $z_1, z_2 \in \R P^2$  the marked points of $\R P^2$ and by $v_1, \ v_2 \colon \R P^2\to \R P^2$ the punctured slides defined on the  marked points $z_1$ and $z_2$, then 
\[
    \mathcal{N}_1^2 \cong \langle {[v_1]} \rangle \times \langle [v_2] \rangle \cong \Z/2 \times \Z/2,
\]
see \cite{Kork02}*{Proof of Corollary 4.6}. Thus, for every $\Zp\leqslant \modnpk$ we have that $\Ima(I)\leqslant \Z/2\times \Z/2$.  
The objective of this part is to prove that $\Ima(I)=\langle [v_1] \cdot [v_2] \rangle \cong \Z/2$ for all $\Zp \leqslant \modnpk$, with $k= 1,2.$ For this purpose we prove the following lemma.


\begin{lem}\label{Lem:Image_of_vi's}
    Let $z_1, z_2 \in \R P^2$ be the marked points of $\R P^2$. Let us denote as $N_1^2 = \R P^2 \setminus \{z_1, z_2 \} $ and let
$$
\pi_1(N_1^2, z_0)=\langle x_1, x_2, d \mid x_1 \cdot x_2 \cdot d^2 =1 \rangle 
$$
be the usual presentation of the group $\pi_1(N_1^2, z_0)$. Then the effect of the isomorphisms induced by the punctured slides $v_i : \R P^2 \to \R P^2 $ for $i=1,2$ and its composition $v_1\circ v_2$ on the generators of $\pi_1( N_1^2, z_0)$ is described in the following table.
\begin{table}[ht]
\begin{center}
	\begin{tabular}{|c|c|c|c|}
	\hline 
	 Gener. & Image under $v_{1*}$ & Image under $v_{2*}$ & Image under $(v_{2}\circ v_1 )_*$ \\ 
	\hline 
	$x_1$ & $d \cdot x_1^{-1}  \cdot d^{-1}$ & $x_2^{-1} \cdot x_1 \cdot x_2 	$	& $d \cdot x_1^{-1} \cdot d^{-1}$  \\ 
	\hline 
	 $x_2$ & $d \cdot x_1  \cdot d^{-1} \cdot x_2 \cdot d \cdot x_1^{-1}  \cdot d^{-1} $ & $d \cdot x_2^{-1}  \cdot d^{-1}$ 	&	$d \cdot x_1 \cdot x_2^{-1} \cdot x_1^{-1} \cdot d^{-1}$\\ 
	\hline 
	 $d$ & $d \cdot x_1$ & $d \cdot x_2$ 	&  $d^{-1}$ \\ 
	\hline 
	\end{tabular} 
\end{center}
\caption{Effect of the punctured slides $v_1$, $v_2$ and $v_1\circ v_2$.}
\label{Tab:Effect_of_Punctured_Slides}
\end{table}
\end{lem}
\begin{proof}
    To see the effect of each punctured slide, consider a disk that contains the marked points and the crosscap inside it. Now, let $\gamma_1,\gamma_2, \delta:[0,1]\to \R P^2$ be loops based on $z_0$ that are representatives of the classes $x_1, x_2, d \in \pi_1(N_1^ 2;z_0)$ respectively as shown in Figure \ref{fig:GeneratorsRP2a}. Let us also take the loops $\beta_1, \beta_ :[0,1] \to \R P ^2$ based on $z_1$ and the loop $\beta_2:[0,1]\to \R P^2$ based on $ z_2$, on which we define the punctured slides $v_1$ and $v_2$ as shown in Figure \ref{fig:GeneratorsRP2b}.

    \begin{figure}[ht]
        \centering
        \begin{subfigure}{0.45\textwidth}
            \centering
            \includegraphics[width=0.65\textwidth]{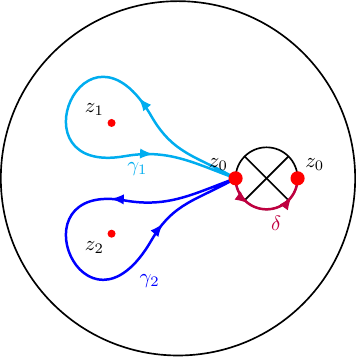}
            \caption{Loops representing the classes $x_1, x_2$ and $d$}
            \label{fig:GeneratorsRP2a}
        \end{subfigure}
        \hspace{1cm}
        \begin{subfigure}{0.45\textwidth}
            \centering
            \includegraphics[width=0.65\textwidth]{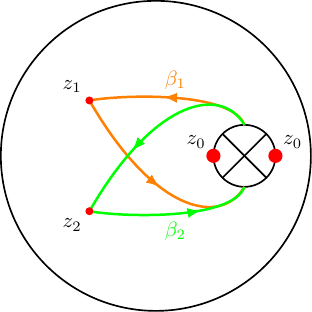}
            \caption{Loops $\beta_1, \beta_2$ on which the punctured slides $v_1$ and $v_2$ are defined}
            \label{fig:GeneratorsRP2b}
        \end{subfigure}
        \caption{Loops generators of the fundamental group of $N_1^2$ and loops used in the punctured slides}
        \label{fig:GeneratorsRP2ab}
    \end{figure}

    Recall that the punctured slide $v_1:\R P^2 \to \R P^2$ is defined by ``pushing'' the points of a regular neighborhood of the curve $\beta_1$ along it. In Figure \ref{Fig:Image_de_x1x2d_under_v1} we see the effect of $v_1$ on each of the loops $\gamma_1$, $\gamma_2$ and $\delta$. We can deduce that the effect of $v_1$ on the generators $x_1$, $x_2$ and $d$ is the one described in the table \ref{Tab:Effect_of_Punctured_Slides}. With similar arguments, it is possible to prove the result for the homeomorphisms $v_2:\R P^2 \to \R P^2$ and $v_1\circ v_2:\R P^2 \to \R P^2$.
    \begin{figure}[hb]
    \centering
    \begin{subfigure}[b]{0.8\linewidth}
        \centering
        \includegraphics[width=\linewidth]{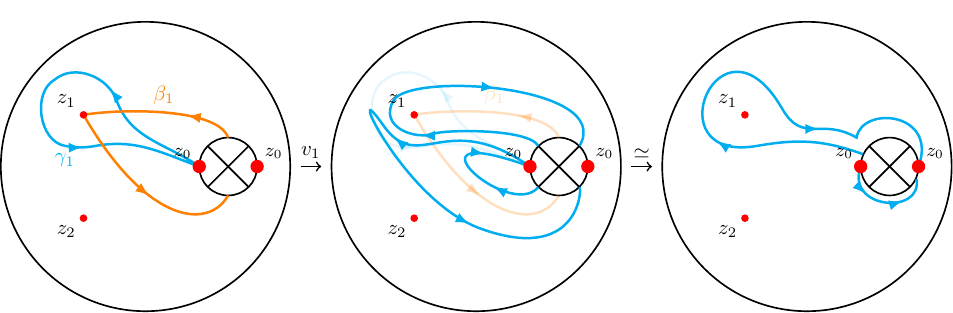}
        \caption{Effect of $v_1$ on $\gamma_1$.}
        \label{Fig:Effect_vgamma1}
    \end{subfigure}
    \vspace{1em} 
    \begin{subfigure}[b]{0.8\linewidth}
        \centering
        \includegraphics[width=\linewidth]{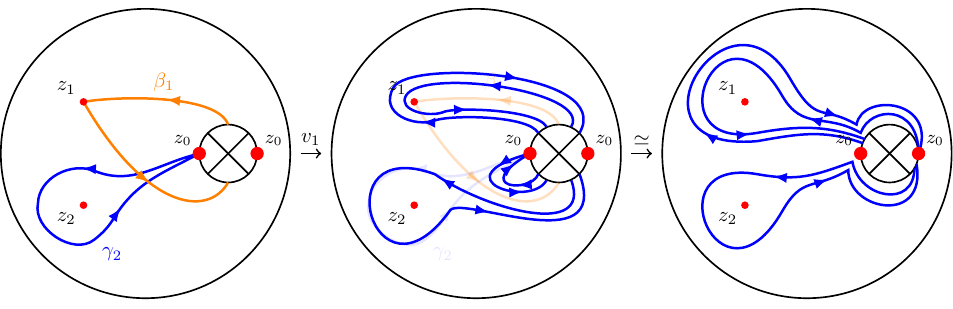}
        \caption{Effect of $v_1$ on $\gamma_2$.}
        \label{Fig:Effect_vgamma2}
    \end{subfigure}
    \vspace{1em} 
    \begin{subfigure}[b]{0.8\linewidth}
        \centering
        \includegraphics[width=\linewidth]{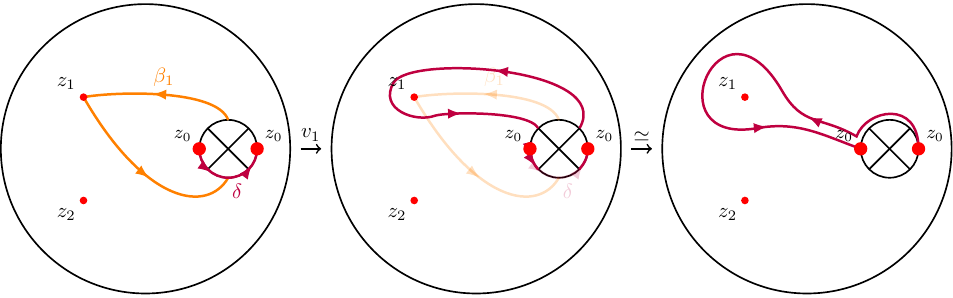}
        \caption{Effect of $v_1$ on $\delta$.}
        \label{Fig:Effect_vdelta}
    \end{subfigure}
    \caption{Effects of $v_1$ on $\gamma_1$, $\gamma_2$, and $\delta$.}
    \label{Fig:Image_de_x1x2d_under_v1}
\end{figure}
\end{proof}

\begin{lem}
    Let $p$ be and odd prime and $\Zp \leqslant \mathcal{N}_p^k$ with $k=1,2$. Then $\Ima (I)=\langle [v_1]\cdot [v_2] \rangle \cong \Z/2.$ In particular, there exist the following group extension  
    $$
        1\to \Zp \to N(\Zp) \to  \Z/2 \to 1.
    $$
\end{lem}
\begin{proof}
    Let $q: N_p \to \R P^2$ be the branched covering obtained from the action of $\Zp$ on the surface $N_p$. Let $\w{z}_1, \w{z}_1\in N_p$ be the branched points of $q$ and $z_1=q(\w{z}_1)$, $z_2=q(\w{z}_2) \in \R P^2$ their images. Denote by $N_p^2 = N_p\setminus \{ \w{z}_1, \w{z}_2 \}$ and $ N_1^2 = \R P^2\setminus \{ z_1,z_2 \} $. Let $z_0\in N_1^2$ be the base point and consider the exact sequence of fundamental groups obtained from the covering $q: N_p^2 \to N_1^2 $ by removing the branched points and their images:
    \[
        1 \to \pi_1(N_p^2;\w{z}_0 ) \xrightarrow{q_*} \pi_1(N_1^2;z_0) =\langle x_1, x_2 , d_1 \mid x_1 \cdot x_2 \cdot d^2 =1 \rangle \xrightarrow{\theta} \Zp \to 1,
    \]
    where $\theta$ is the monodromy action of the covering $q:N_p^2\to N_1^2$. Now, recall that $I:N(\Zp)\to \Lmcg(\R P^2;2)$ where $\Lmcg(\R P^2;2)$ is the subgroup of mapping class group $\mcg(\R P^2;2)$ that consist of mapping classes that has a representative which lifts under the covering $q: N_p^2 \to N_1^2$. When $k=1,2$ the image of $I$ is contained in $\pmcg(\R P^2;2)\cong \Z/2\times \Z/2$. Moreover,  $\pmcg(\R P^2;2)\cong \langle [v_1] \rangle \times \langle [v_2] \rangle$ where $v_i$ is the punctured slide along a loop starting in the point $z_i$ for $i=1,2$. Thus, we would like to know if the mapping classes of the homeomorphism $v_1$, $v_2$ and $v_1\circ v_2$ lifts to a homemorphism in $N_p$. For basic theory of coverings a homeomorphism $y:N_1^2\to N_1^2 $ such that $y(z_0)=z_0$ lifts under $q:N_p^2\to N_1^2$ if and only if  $y_* (\ker(\theta))=\ker(\theta)$. Our objective is to prove the following:
    \begin{align*}
        v_{i*}(\ker(\theta)) & \neq \ker(\theta) & & \text{for } & i= 1,2 &  & \text{and} & & (v_1 \circ v_2)_* (\ker(\theta))= \ker(\theta).
    \end{align*}
    Notice that a loop $\gamma_i$ round the point $z_i\in \R P^2$ representative  of the generator $x_i$ does not lift a loop in $N_p^2$, since locally around of the point $z_i$ the covering $q:N_p\to \R P^2$ is modeled locally like a rotation (posibly with a composition of a reflection). Then $\theta(x_i)\neq 0$ for $i=1,2$. Consider the element $\beta=x_1 x_2 d^2\in \pi_1(N_1^2;z_0) $. We see that $\theta(\beta)=0$, but $\theta((v_1)_*(\beta))= \theta(x_2)+2\cdot\theta(d) = -\theta(x_1) \not\equiv 0 \mod (p) $ according to Lemma \ref{Lem:Image_of_vi's}. Consequently $(v_1)_*(\ker(\theta))\neq \ker(\theta)$. And also, we have that $\theta((v_2)_*(\beta))=-\theta(x_2)\not\equiv 0 \mod(p)$. Hence the mapping classes $[v_1],[v_2]\not\in \Lmcg(\R P^2;2)$. On the other hand, let be $w\in \pi_1(N_1^2;z_0  )$, then $w$ has an expression of the following form
    \[
        w=\prod\limits_{i=1}^{\alpha} x_1^{n_i} x_2^{m_i} d^{k_i}.
    \]
    Applying $\theta$ to $w$ we obtain
    \[
    \theta(w)=\sum\limits_{i=1}^{\alpha} (n_i \cdot \theta(x_1) + m_i \cdot \theta(x_2) + k_i \cdot\theta(d)).
    \]
    By Lemma \ref{Lem:Image_of_vi's}, we have 
    $$
    \theta((v_1\circ v_{2})_* (w))= \sum\limits_{i=1}^{\alpha} (n_i \cdot [-\theta(x_1)] + m_i \cdot [-\theta(x_2)] + k_i \cdot[-\theta(d)])=-\theta(w).
    $$
    Therefore,
    $w\in \ker(\theta)$ if and only if $\theta((v_1\circ v_{2})_* (w))\in \ker(\theta)$.
\end{proof}

\begin{proof}[Proof of Theorem \ref{Thm:Main:Nor:Npk}]
    By the previous lemma, we have an extension 
    $$
        1\to \Zp \to N(\Zp) \to  \Z/2 \to 1,
    $$
    thus $N(\Zp)$ is a group of order $2p$. Thus, the only possibilities for $N(\Z_p)$ are the following two
    \begin{align*}
        N(\Zp ) & \cong \Zp\times \Z/2 & \text{or} & & N(\Zp) & \cong D_{2p}.
    \end{align*}
    On the other hand, by \cite{CJX24Per}*{Theorem 2} the pure mapping class group $\modnpk$ has $p$-period equal to $4$. Using this fact and seeing that the $p$-period of $\Zp\times \Z/2$ is equal to $2$ and the $p$-period of $D_{2p}$ is equal to $4$, we conclude that $N(\Zp)\cong D_{2p}$, proving the result.
\end{proof}


\section{The \texorpdfstring{$p$}{p}-primary component of the Farrell cohomology of \texorpdfstring{$\mathcal{N}_p^k$}{Ngk}}
\label{Sec:FarrellNpk}

In this section, we apply the results of this paper to prove Theorem \ref{Thm:Main:FarrellNpk}, which provides an explicit computation of the $p$-primary component of the Farrell cohomology of the pure mapping class group $\modnpk$. 

\begin{proof}[Proof of Theorem \ref{Thm:Main:FarrellNpk}]
Theorem \ref{Thm:Main:Torsion} shows that $\modnpk$ contains $p$-torsion for $k=1,2$, while for $k \geqslant 3$ it does not have $p$-torsion. Thus, 
\[ 
    \Farrp{*}{\modnpk}=0 \quad \text{for} \quad k\geqslant 3.
\]
For  $k=1,2$, by \cite{CJX24Per}*{Theorem 1} we have that $\Farrp{*}{\modnpk}$ has $p$-periodic cohomology and Brown's formula reduce its calculation:
\[ 
    \Farrp{*}{\modnpk} \cong \prod_{\Zp \in S } \Farrp{*}{N(\Zp)},
\]
where $S$ is a set of representatives of the conjugacy classes of subgroups of $\modnpk$ of order $p$, $N(\Zp)$ is the normalizer of $\Zp$ in $\modnpk$. By Corollary \ref{Cor:Number_of_Conjugacy_clases_Npk}, there are $(p-1)/2$ conjugacy classes of subgroups of order $p$ for $\modnpk$, classified by two-tuples $(1\mid\beta_2)$ up to congruence equivalence. Furthermore, Theorem \ref{Thm:Main:Nor:Npk} establishes that $N(\Zp) \cong D_{2p}$. Thus, the result follows directly from Brown’s formula.
\end{proof}

Applying Theorem \ref{Thm:Main:FarrellNpk}, and knowing that $\vcd(\modnpk)=2p-4+k$, we deduce the following result.

\begin{cor}
    Let $p$ be an odd prime. Then, for $i \geqslant 2p-4+k$, we have:
    \begin{align*}
        H^{i}({\modnpk};\Z)_{(p)}=& \begin{cases}
        \left( \Zp \right)^{\tfrac{p-1}{2}} & i\equiv 0 \mod 4, \\
         0                                  & \text{otherwise},
        \end{cases}    
    & \text{for } & k=1,2; \\
    \\
        H^{i}({\modnpk};\Z)_{(p)} = & \ \ 0 \quad \text{for } k\geqslant 3.
    \end{align*}
\end{cor}

Finally, the method described in this article can be applied to determine the $p$-primary component of the Farrell cohomology for the pure mapping class groups $\mathcal{N}_g^k$ for $g\geqslant 3$ and $k\geqslant 1$, as present in Theorem \ref{Thm:Main:FarrellNpk}. To apply this approach, Theorem \ref{Thm:Main:Torsion} provides conditions to have $p$-torsion and Theorem \ref{Thm:Main:CClass:SK} give us the number and classification of conjugacy classes via their corresponding $t$-tuple. The main difficulty to provide more examples of the $p$-primary component of $\Farr{*}{\modn}$ lies in understanding the extension of Proposition \ref{Prop:Extension_of_Normalizers}:
    \[
        1\to \Zp \to N(\Zp) \to \Ima(I) \to 1,
    \]
which provide information to determine the cohomology of the normalizers. Thus, the approach developed here opens new directions for future work to provide new examples and a better understanding of the cohomology of $\modn$ in the context of non-orientable surfaces.

\section*{Acknowledgements}
This work started as part of the Ph.D. dissertation of the author. He would like to thank Rita Jiménez Rolland and Miguel A. Xicoténcatl for their support and valuable comments while this work was conducted. The author was funded by CONAHCyT through the program Estancias Posdoctorales por México.

\bigskip

\bibliographystyle{amsalpha}

\bibliography{References}

\end{document}